\documentclass[11pt]{article}

\usepackage[a4paper, margin=2.5cm]{geometry}
\usepackage{amsmath, amssymb, amsthm}
\usepackage{mathtools}
\usepackage{bm}          
\usepackage{color}
\usepackage{graphicx}
\usepackage{caption}
\usepackage{hyperref}
\usepackage{enumitem}
\usepackage{booktabs}
\usepackage{multirow}
\usepackage{algorithm}
\usepackage{algpseudocode}
\usepackage{natbib}
\usepackage{microtype}
\usepackage{xcolor}
\usepackage{pgf}          
\usepackage{orcidlink}    

\newtheorem{theorem}{Theorem}[section]
\newtheorem{proposition}[theorem]{Proposition}

\newtheorem{remark}[theorem]{Remark}
\newtheorem{assumption}[theorem]{Assumption}

\newcommand{\R}{\mathbb{R}}
\newcommand{\E}{\mathbb{E}}
\newcommand{\bx}{\mathbf{x}}
\newcommand{\bk}{\mathbf{k}}
\newcommand{\bK}{\mathbf{K}}

\newcommand{\IK}{\mathcal{I}_{\bK}}

\newcommand{\Kmax}{K_{\max}}

\newcommand{\bmu}{\boldsymbol{\mu}}
\newcommand{\bomega}{\boldsymbol{\omega}}
\newcommand{\balpha}{\boldsymbol{\alpha}}

\newcommand{\calA}{\mathcal{A}}

\newcommand{\norm}[1]{\left\|#1\right\|}

\renewcommand{\Re}{\operatorname{Re}}
\newcommand{\TT}{\mathrm{TT}}
\newcommand{\CTT}{\mathrm{CTT}}

\begin{document}

\title{%
Fourier--cosine Tensor Trains for Density Recovery and Expectation Calculation}

\author{%
  Auke Schaap\thanks{Delft Institute of Applied Mathematics, Delft University of
  Technology, Delft, the Netherlands.}
  \and
  Fang Fang\,\orcidlink{0000-0003-2789-8255}\thanks{FF Quant Advisory B.V., Utrecht, the Netherlands, and Delft
  Institute of Applied Mathematics, Delft University of Technology, Delft, the
  Netherlands. Corresponding author:
  \href{mailto:f.fang@tudelft.nl}{f.fang@tudelft.nl}.}
}

\date{\today}

\maketitle

\begin{center}
  \small Preprint. Submitted to a journal; not peer-reviewed.
\end{center}


\begin{abstract}

The Fourier--cosine (COS) method of Fang and Oosterlee (2008) 
recovers densities and computes expectations semi-analytically from 
characteristic functions (ch.f.s). In higher dimensions, both expansion 
synthesis and COS coefficient-tensor construction scale exponentially with
 dimension. We address this curse of dimensionality with two tensor-train 
 (TT) methods. A straightforward TT decomposition of the COS coefficient 
 tensor gives COS-TT, which removes the exponential cost of online synthesis
  but not of offline decomposition. Our main contribution, COS-TT-CHF, 
  instead compresses a ch.f. sample tensor and maps it to the original 
  Fourier--cosine coefficient tensor through an alternative COS representation. 
  Apart from the model-dependent cost of one ch.f. evaluation, both offline
   and online costs scale linearly with dimension, while error grows only algebraically.
COS-TT-CHF introduces two additional errors: frequency-domain truncation
 and discretization errors. Both are controlled by parameter-selection rules 
 from theoretical error analysis. We further propose an approach for expectations
  of nonseparable functions of additive scalar aggregates by computing the ch.f.s 
  of the aggregates using tensorized COS methods, 
  reducing the original multidimensional problem to a one-dimensional COS calculation.
   Closed-form formulas are derived for the TT-contraction integrals arising
   when the aggregate is a weighted sum, as in European basket option pricing.
Experiments under geometric Brownian motion (GBM) and variance gamma (VG) show that,
with the parameter-selection rules, COS-TT-CHF remains 
accurate through 150 dimensions for GBM and 100 for VG within a 30-minute 
offline computation budget on a laptop. As a by-product, our Fourier--cosine
TT constructions yield cosine-basis functional TTs (FTTs) under weaker assumptions
 than in spectral FTT literature.

\end{abstract}

\section{Introduction}

\label{sec:intro}

Recovering joint densities and evaluating expectations of dependent random
variables are central to uncertainty quantification, reliability analysis,
statistical signal processing, stochastic modelling, and mathematical finance.
Direct tensor-product discretizations incur a computational cost that grows
exponentially with dimension for both tasks. In many models, including infinitely
divisible distributions and affine processes \citep{duffie2002affine}, the density
is unavailable in closed form, whereas the joint characteristic function (ch.f.)
is known analytically or can be evaluated inexpensively.

The Fourier--cosine (COS) method exploits this asymmetry. Introduced by
\citet{fang2009novel} for one-dimensional option pricing, it was later extended
and analysed in multidimensional settings
\citep{ruijter2012two,ruijter2015application,junike2025characteristic}.
The method recovers a density through a truncated cosine expansion whose
coefficients are directly linked to the ch.f., and in turn expresses
the expectation of a target function as the inner product of two sets 
of cosine coefficients: one associated withthe density and the other with
 the target function. 
In multiple dimensions, however, the COS formula faces
two compounding manifestations of the curse of dimensionality: the coefficient
tensor has exponentially many entries, and computing each entry requires summing
ch.f.\ values over exponentially many sign combinations.

Low-rank tensor representations offer a remedy, as explored by existing works
in literature (Section \ref{sec:related}). A companion paper by 
members of our group \citep{MastFang2026} develops a canonical-polyadic 
realization. Here we adopt
the tensor-train (TT) format \citep{oseledets2011tensor}, which represents a
multidimensional tensor as a chain of low-dimensional cores. We develop two
methods, COS-TT and COS-TT-CHF, which yield the same tensorized representation
of the Fourier--cosine expansion through different routes to the TT cores
but construct the TT cores through different routes.
We refer to them collectively as \emph{tensorized COS methods}.
COS-TT-CHF is our main innovation and contribution, while COS-TT serves as a 
simpler complementary method and benchmark.

\textbf{COS-TT} is a direct approach that applies TT decomposition to
compress the coefficient tensor of the multidimensional COS formula. We
refer to this tensor as the \emph{COS coefficient tensor} to distinguish
it from the Fourier--cosine coefficient tensor that it approximates.
Compression removes the curse of dimensionality from expansion synthesis,
enabling efficient online density recovery and expectation computation.
Offline construction, however, still requires access to individual tensor
entries, each of which involves exponentially many ch.f.\ evaluations. We
nevertheless provide a full derivation of COS-TT: it remains useful in
moderate dimensions, serves as a benchmark, is simpler to implement, and
inherits the standard COS error control established in the literature.

\textbf{COS-TT-CHF}, our main innovation and contribution, instead first
compresses a tensor of ch.f.\ samples, each requiring only a single ch.f.
evaluation, and then applies a rank-preserving transformation to obtain
the Fourier--cosine coefficient cores. With all other discretization and
rank parameters fixed, both TT construction and online computation scale
linearly with dimension, apart from the model-dependent cost of each
ch.f.\ evaluation. The method introduces two additional errors:
frequency-domain truncation error and discretization error, controlled by
the truncation range and the number of quadrature points, respectively.
Our error analysis yields parameter-selection rules under which
COS-TT-CHF achieves accuracy comparable to COS-TT. For both methods, we
also prove that density recovery does not amplify coefficient-level
approximation error.

Beyond density recovery, we develop an approach for evaluating
expectations of nonseparable functions of additive scalar aggregates.
Existing approaches apply TT decompositions to both the target function
and the density. We instead contract the tensorized COS representation of
the density to obtain the ch.f.\ of the aggregate, reducing the
multidimensional problem to a one-dimensional one to which the standard
one-dimensional COS method can be applied. For weighted-sum aggregates,
as in European basket option pricing, closed-form expressions for the
contraction integrals yield a semi-analytical formula.

Experiments under geometric Brownian motion (GBM) and variance gamma (VG)
models demonstrate algebraic growth of the error with dimension. With the
proposed parameter-selection rules, COS-TT-CHF remains accurate up to
dimension $150$ for GBM and $100$ for VG within a $30$-minute offline
computation budget on a laptop. It thus effectively breaks the curse of
dimensionality for the density-recovery and expectation-computation
problems considered here. More broadly, COS-TT-CHF provides a route to
extending Fourier-based methods to problems previously limited by high
dimensionality.

As a by-product, our derivations provide an alternative route to
cosine-basis functional tensor trains (FTTs) under weaker assumptions than
existing spectral derivations, as discussed in
Section~\ref{subsec:duality}. Section~\ref{sec:related} relates these
contributions to the existing literature.

\section{Related work}
\label{sec:related}

Low-rank tensor methods provide powerful tools for high-dimensional
approximation, integration, and option pricing
\citep{DolgovSavostyanov2019,glau2020low,kastoryano2022highly,
sakurai2025learning}. Their relevance to the present work turns on two
distinctions: assumptions and methodology. In this paper, the ch.f.\ is
known but the density need not be, and for expectation calculations we
compress only the density, not the target function.

\paragraph{Tensor integration and parametric pricing.}

\citet{DolgovSavostyanov2019} develop parallel TT-cross interpolation
(Section~\ref{subsec:TT-algorithms}) for high-dimensional integration.
The method requires pointwise evaluations of the integrand, which it
samples adaptively rather than assembling on a full tensor-product grid.
Applied directly to an expectation written as an integral of a target
function against a density, it therefore requires pointwise access to the
density. The integration algorithm can equally be used with other
evaluable representations of the integrand, including suitable Fourier
formulations, but it does not itself provide the ch.f.-to-density
construction studied here.

\citet{glau2020low} combine tensor approximation with Chebyshev
interpolation for parametric option pricing. They interpolate prices in
model and payoff parameters on a tensorized Chebyshev grid and recover
the coefficients by TT completion, with each tensor entry produced by a
reference pricer. Their tensor is indexed by model and payoff parameters
rather than by the underlying risk factors, so it compresses the
parameter dependence of a price rather than the expansion of a density.
It therefore serves as an acceleration method for pricing individual
products across parameter configurations. The two constructions address
different stages of the computation and may be complementary.

\paragraph{Tensor pricing from ch.f.\ samples.}

The approaches most closely related to ours are those of
\citet{kastoryano2022highly} and \citet{sakurai2025learning}, which also work
with ch.f.\ samples without requiring pointwise density evaluations. Their
Fourier pricing formulas additionally require evaluable payoff transforms
and suitable complex shifts of the integration contour.

\citet{kastoryano2022highly} express the discretized Fourier price as an
inner product of sampled ch.f.\ and payoff-transform values, approximating
both tensors in TT format by cross interpolation. For their minimum-of-assets
payoff, the transform couples the frequencies through their sum and cannot
be written as a single separable product; the shifted integration contour
avoids its poles. In their main benchmark, the chosen payoff bond
dimensions are twice those of the ch.f.\ tensor. Their separate three-asset
study  fixes the payoff bond dimension at $30$ to make its
approximation error negligible and isolate the error from the ch.f.\ tensor.
They report results for up to $15$ assets and encounter TT-cross convergence
difficulties beyond this dimension. 

\citet{sakurai2025learning} extend Fourier--TT pricing to support repeated
pricing as model parameters vary. In their construction, the ch.f.\ tensor depends
jointly on frequency and parameters, while the payoff-transform tensor depends
only on frequency. Offline contraction over the frequency indices, followed
by SVD compression, produces a TT representation of the price over the
parameter grid, enabling rapid online evaluation. Their min-call experiments vary either
volatilities or initial asset prices under the Black--Scholes model and reach
$11$ assets. This representation can be reused across model configurations,
although incorporating parameter dependence may increase the approximation
demands relative to a fixed-model calculation. 

Our approach to expectation calculation avoids tensorizing the payoff
altogether. For targets of the form $G(H)$, where $H$ is an additive
scalar aggregate of coordinate-wise functions, the exponential appearing
in the ch.f.\ of $H$ factorizes across coordinates. Substituting a
tensorized COS approximation of the joint density into the integral
defining this ch.f.\ reduces its evaluation to efficient TT contractions
involving one-dimensional integrals. Once the ch.f.\ of $H$ is obtained,
the original multidimensional expectation reduces to a one-dimensional
expectation of $G(H)$, to which the standard one-dimensional COS method
can be applied. When the aggregate is a weighted sum, as in European
basket option pricing, Theorem~\ref{thm:inner-integrals} provides
closed-form expressions for the contraction integrals. This reduction
applies to the stated class of additive aggregates; we do not claim an
equivalent reduction for arbitrary multivariate targets.

\paragraph{Spectral tensor representations and the COS construction.}

The low-rank spectral representations closest in form to our construction
are FTTs \citep{bigoni2016spectral} and CPD-based Fourier or cosine
constructions \citep{Kargas2021,amiridi2022lowrank}. As discussed in
Section~\ref{subsec:duality}, our derivations yield a cosine-basis FTT
through an alternative construction that requires only the ch.f.\ and
converges to the density under square integrability alone. The existing
spectral FTT derivation, by comparison, assumes H"older continuity and
Sobolev regularity.

CPD can be less accurate in practice \citep{MastFang2026}, and for tensors
of order at least three, sets of bounded CP rank are generally nonclosed,
so a best fixed-rank CP approximation need not exist
\citep{de2008tensor}. The companion work \citep{MastFang2026} develops a
CP realization of the tensorized COS principle for credit-exposure
calculations. Although the resulting representation is fully separable,
the offline decomposition becomes a bottleneck beyond $20$ dimensions.
Here, we develop the TT realization and introduce COS-TT-CHF to overcome
this offline bottleneck.

\paragraph{Relationship to earlier work.}

Both COS-TT and COS-TT-CHF were first derived by the authors of the
present paper and documented in an MSc thesis \citep{schaap2025costt}.
The thesis timeframe did not permit a comprehensive study, particularly
of COS-TT-CHF, which was developed near the end of the thesis period.
The methodology presented here is unchanged from those earlier
derivations; the contribution of the present paper is to develop them
into a complete framework, including systematic derivations, extensive
error analysis, theoretical proofs, practical parameter-selection rules,
and comprehensive numerical tests in substantially higher dimensions.

While this paper was being prepared, \citet{arenstein2026costtchf}
independently reimplemented COS-TT-CHF for option pricing, following the
construction documented in our thesis without methodological changes.
We refer to their work for complementary option-pricing benchmarks of
COS-TT-CHF and COS-TT.

\section{Background and notation}
\label{sec:prelim}

\subsection{The COS method}
\label{subsec:COS}

Let $f$ be a continuous probability density on $\R$. It is integrable, so its
Fourier transform, the ch.f., is well defined and continuous:
\[
    \varphi(\omega)
    =
    \int_{\R} f(x)e^{i\omega x}\,dx .
\]
Assume also that $f$ is piecewise $C^1$ on
$[a,b]\subset\R$. Then its Fourier--cosine series converges uniformly,
and hence pointwise, to $f$ on $[a,b]$:
\begin{equation}
    \label{eq:cosine-series-1d}
    f(x)
    =
    \sum_{k=0}^{\infty}{}'
    F_k
    \cos\!\left(
        \frac{k\pi}{b-a}(x-a)
    \right),
    \qquad
    F_k
    =
    \frac{2}{b-a}
    \int_a^b
    f(x)
    \cos\!\left(
        \frac{k\pi}{b-a}(x-a)
    \right)dx ,
\end{equation}
where $\sum'$ indicates that the term with $k=0$ is multiplied by one half.
The COS method truncates the series after $K$ terms and approximates the
coefficients directly from the ch.f.:
\begin{equation}
    \boxed{
    \label{eq:COS-coeff}
    F_k
    \approx
    A_k
    :=
    \frac{2}{b-a}
    \Re\!\left\{
        \varphi\!\left(\frac{k\pi}{b-a}\right)
        \exp\!\left(
            -i\frac{k\pi a}{b-a}
        \right)
    \right\}.}
\end{equation}
The two steps together give the COS formula for density recovery,
\begin{equation}
    \label{eq:COS-approx-1d}
    f(x)
    \approx
    \hat{f}(x)
    :=
    \sum_{k=0}^{K-1}{}'
    A_k
    \cos\!\left(
        \frac{k\pi}{b-a}(x-a)
    \right).
\end{equation}
This construction extends to multiple dimensions.
Let $\mathbf{D}$ denote a finite hyperrectangle
\begin{equation}
    \label{eq:domain-D}
    \mathbf{D}
    :=
    D_1\times\cdots\times D_d
    \subset\R^d,
    \qquad
    D_j:=[l_j,u_j].
\end{equation}
Let $\mathbf X=(X_1,\ldots,X_d)$ have a continuous density $f_{\mathbf X}$ on
$\R^d$ whose restriction to $\mathbf D$ belongs to $L^2(\mathbf D)$. Its joint
ch.f.\ is
\[
    \varphi_{\mathbf X}(\bomega)
    =
    \int_{\R^d}f_{\mathbf X}(\bx)e^{i\bomega\cdot\bx}\,d\bx.
\]

\begin{remark}[Truncation range in the physical domain]
\label{rem:domain-rule}
For the physical-domain truncation we follow the cumulant-based rule of
\citet{fang2009novel}, applied coordinatewise:
\begin{equation}
    \label{eq:domain-X}
    l_j
    =
    c_{1,j} - \lambda_1\sqrt{c_{2,j}+\sqrt{c_{4,j}}},
    \qquad
    u_j
    =
    c_{1,j} + \lambda_1\sqrt{c_{2,j}+\sqrt{c_{4,j}}},
\end{equation}
where $c_{n,j}$ is the $n$-th cumulant of $X_j$ and $\lambda_1>0$ is a
dimensionless multiplier, selected in Section~\ref{subsec:num-K-rank}. The
range is therefore centred at the marginal mean
$c_{1,j}$, with half-width $\ell_j = \lambda_1\sqrt{c_{2,j}+\sqrt{c_{4,j}}}$
set by the marginal standard deviation $\sqrt{c_{2,j}}$ and corrected for the
fourth cumulant, so that little of the mass of $f_{\mathbf X}$ falls outside
$\mathbf{D}$.
We write $\lambda_1$ rather than the $L$ of \citet{fang2009novel} to avoid a
clash with the $L^1$ and $L^2$ norms used throughout.
\end{remark}

\begin{remark}[Norms and types of convergence]
\label{rem:norms}
The one- and multidimensional statements use different function spaces. In one
dimension, we assume $f\in L^1(\R)$ and piecewise $C^1$, following
\citet{fang2009novel}, so that \eqref{eq:cosine-series-1d} converges uniformly.
In $d$ dimensions, we assume only $f_{\mathbf X}\in L^2(\mathbf D)$ and
interpret the expansion in $L^2$. The cosine basis is orthonormal in this space,
Parseval and Plancherel convert coefficient errors into density errors, and the
TT format is defined on
$L^2(\mathbf D)=L^2(D_1)\otimes\cdots\otimes L^2(D_d)$. Accordingly, all density
errors below are measured in $L^2(\mathbf D)$. Section~\ref{subsec:COS-TT-CHF}
additionally assumes $f_{\mathbf X},\varphi_{\mathbf X}\in L^1(\R^d)$ to justify
Fourier inversion and the interchange of integrals. Convergence rates require stronger assumptions,
namely polynomial ch.f.\ decay of order $p>d/2$ and COS-admissibility; see
\citet[Thm.~3.7 and Prop.~A.1]{junike2025characteristic}.
\end{remark}

Throughout the paper, $\Phi_{\bk}(\bx)$ denotes the product of univariate
cosine basis functions:
\begin{equation}
    \label{eq:product-basis}
    \Phi_{\bk}(\bx)
    =
    \prod_{j=1}^{d}
    \phi^{(j)}_{k_j}(x_j),
    \qquad
    \phi^{(j)}_{k_j}(x_j)
    =
    \beta^{(j)}_{k_j}
    \cos\!\left(
        \frac{k_j\pi}{u_j-l_j}(x_j-l_j)
    \right),
\end{equation}
where
$\beta^{(j)}_{k_j}=\sqrt{1/(u_j-l_j)}$ for $k_j=0$ and
$\beta^{(j)}_{k_j}=\sqrt{2/(u_j-l_j)}$ for $k_j\geq 1$.
The normalization absorbs the one-half factor of the primed sum in
\eqref{eq:cosine-series-1d}, so ordinary sums are used from here on. It also
makes $\{\Phi_{\bk}\}_{\bk\in\mathbb N_0^d}$ a complete orthonormal basis of
$L^2(\mathbf D)$, so every $f_{\mathbf X}\in L^2(\mathbf D)$ admits a
Fourier--cosine expansion in $L^2$:
\begin{equation}
    \label{eq:multD_Fourier_cosine}
    \lim_{\bK\to\infty}
    \norm{
        f_{\mathbf X}(\bx)
        -
        \tilde{f}_{\bK}(\bx)
    }_{L^2(\mathbf D)}
    =
    0,
    \qquad
    \tilde{f}_{\bK}(\bx)
    :=
    \sum_{\bk}^{\bK}
    \mathcal{F}_{\bk}\Phi_{\bk}(\bx).
\end{equation}
Here $\bk=(k_1,\ldots,k_d)\in\mathbb N_0^d$ is the multi-index of expansion
indices, $\bK=(K_1,\ldots,K_d)$ collects the term counts, $\bK\to\infty$ means
$K_j\to\infty$ for every $j$, and
\[
    \sum_{\bk}^{\bK}
    :=
    \sum_{k_1=0}^{K_1-1}
    \cdots
    \sum_{k_d=0}^{K_d-1}.
\]
The Fourier--cosine coefficient tensor $\mathcal{F}$ in \eqref{eq:multD_Fourier_cosine}
is the $d$-dimensional counterpart of $F_k$ in
\eqref{eq:cosine-series-1d}. Its entries are defined by
\begin{equation}
    \label{eq:F-tensor}
    \mathcal{F}_{\bk}
    :=
    \left\langle
        f_{\mathbf X},
        \Phi_{\bk}
    \right\rangle_{L^2(\mathbf{D})}
    =
    \int_{\mathbf{D}}
    f_{\mathbf X}(\bx)
    \Phi_{\bk}(\bx)\,d\bx.
\end{equation}

The multidimensional counterpart of $A_k$ in \eqref{eq:COS-coeff} is the COS
coefficient tensor $\mathcal{A}$, whose entries are computed directly from
ch.f.\ values.
Elementwise, $\mathcal{A}$ is defined by
\begin{equation}
    \label{eq:multi-COS-coeff}
    \mathcal{A}_{\bk}
    =
    \frac{
        \prod_{j=1}^{d}\beta^{(j)}_{k_j}
    }{2^{d-1}}
    \sum_{\mathbf{s}\in\mathcal{S}_d}
    \Re\!\left\{
        \exp\!\left(
            -i
            \sum_{j=1}^{d}
            \frac{s_jk_j\pi l_j}{u_j-l_j}
        \right)
        \varphi_{\mathbf{X}}\!\left(
            \frac{s_1k_1\pi}{u_1-l_1},
            \ldots,
            \frac{s_dk_d\pi}{u_d-l_d}
        \right)
    \right\},
\end{equation}
where
$\mathcal{S}_d=\{(1,s_2,\ldots,s_d):s_i\in\{-1,+1\}\}$
is the set of sign vectors of length $d$ whose first entry is fixed to
$+1$, so that $|\mathcal{S}_d|=2^{d-1}$. The sum over
$\mathcal{S}_d$ replaces the single real part in the one-dimensional
formula \eqref{eq:COS-coeff}: the multidimensional Fourier transform
produces a cosine of a signed sum, whereas $\Phi_{\bk}$ is a product
of univariate cosines.
Replacing $\mathcal{F}$ by $\mathcal{A}$ gives the multidimensional COS
approximation of a density from its ch.f.:
\begin{equation}
    \label{eq:fhat}
    \tilde{f}_{\bK}(\bx)
    \approx
    \hat{f}_{\bK}(\bx)
    :=
    \sum_{\bk}^{\bK}
    \mathcal{A}_{\bk}\Phi_{\bk}(\bx),
    \qquad
    \bx\in\mathbf{D}.
\end{equation}
The curse of dimensionality in evaluating \eqref{eq:fhat} arises in two places.
First, the COS coefficient tensor contains
\(
    \prod_{j=1}^{d}K_j
\)
entries, so directly contracting it with the tensor product of cosine
basis functions requires work proportional to $\prod_{j=1}^{d}K_j$.
In the isotropic case, $K_j=K$, this gives $\mathcal{O}(K^d)$
complexity. Second, each entry of $\mathcal{A}$ combines $2^{d-1}$ sign patterns
of the ch.f.\ and therefore costs $\mathcal{O}(2^{d-1})$ to evaluate. Forming
the entire tensor requires
\(
    \mathcal{O}\!\left(
        2^{d-1}\prod_{j=1}^{d}K_j
    \right),
\)
or $\mathcal{O}(2^{d-1}K^d)$ in the isotropic case.

\paragraph{Notation simplification.} Henceforth, the expansion-term count $\bK$ is fixed,
so we omit the subscript $\bK$. We also abbreviate $f_{\mathbf X}$ by $f$,
$\varphi_{\mathbf X}$ by $\varphi$, and write
$\tilde{f}$ and $\hat{f}$ for $\tilde{f}_{\bK}$ and
$\hat{f}_{\bK}$ in \eqref{eq:multD_Fourier_cosine} and
\eqref{eq:fhat}.

\subsection{TT decomposition}
\label{subsec:formats}
\label{subsec:TT}
\label{subsec:TT-algorithms}

We use the TT format of \citet{oseledets2011tensor}.

Let $\mathbb F$ denote either $\R$ or $\mathbb C$. An order-$d$ tensor
$\mathcal{T}\in\mathbb F^{n_1\times\cdots\times n_d}$ admits a TT representation if
\begin{equation}
    \label{eq:TT-discrete}
    \mathcal{T}[i_1,i_2,\ldots,i_d]
    =
    \mathbf{T}_1(i_1)\mathbf{T}_2(i_2)\cdots\mathbf{T}_d(i_d),
    \qquad
    i_j=1,\ldots,n_j,
\end{equation}
where
$\mathbf{T}_j(i_j)\in\mathbb F^{r_{j-1}\times r_j}$
is the $i_j$-th matrix slice of a three-dimensional core
$T_j\in\mathbb F^{r_{j-1}\times n_j\times r_j}$ and
$r_0=r_d=1$.
Here $n_j$ is the coordinate size, $i_j$ the corresponding index, and
$\mathbf{r}=(r_1,\ldots,r_{d-1})$ the TT ranks. In Section~\ref{sec:method},
$n_j=K_j$ for the COS coefficient tensor in COS-TT and $n_j=M_j$ for the
sampled ch.f.\ tensor in COS-TT-CHF.

Two properties make the format useful here. First, if
$n=\max_j n_j$ and $r=\max_j r_j$,
storage scales as
$\mathcal{O}(dnr^2)$,
compared with $\mathcal{O}(n^d)$ for the full tensor. Sets of tensors of bounded
TT rank are closed, so best fixed-rank approximations exist
\citep[Cor.~2.4]{oseledets2011tensor}.
Second, contracting a TT tensor with a separable (rank-one) tensor factorizes
across dimensions at cost $\mathcal{O}(dnr^2)$, hence linearly in the number of
dimensions for bounded coordinate size and TT rank. The later derivations use
the functional counterpart, Theorem~\ref{thm:FTT-inner-product}.

\paragraph{Decomposition algorithms.}
The methods in Section~\ref{sec:method} depend only on the TT representation
and apply to any algorithm that computes the cores. The choice of
algorithm affects performance and contributes an approximation error, but it
does not alter the tensorized COS formula.

The classical TT-SVD algorithm \citep{oseledets2011tensor} performs successive
singular value decompositions of tensor unfoldings. It is deterministic and
quasi-optimal, but it requires access to the full tensor, which rules it out
here.

Cross approximation \citep{OseledetsTyrtyshnikov2010} instead builds the TT
representation from selected tensor entries, so the full tensor is never
formed. It rests on the pseudo-skeleton principle from matrix approximation and
chooses interpolation sets from tensor fibres. TT-Cross selects these sets by a
maximum-volume principle, whose quasi-optimality factor does not grow
exponentially with $d$ \citep{Savostyanov2014,qin2022error}. Since computing
maximum-volume submatrices is NP-hard, this is an existence result only, and
practical algorithms carry no comparable guarantee.

The density-matrix renormalization group (DMRG) strategy treats two
neighbouring dimensions together as a superblock, determines an intermediate rank,
and then splits the superblock into two TT cores while sweeping through the
train
\citep{white1992density,schollwock2011density}. The DMRG-greedy variant
\citep{Savostyanov2014} combines this rank-adaptive construction with greedy
interpolation, so the ranks need not be prescribed in advance.

\paragraph{The algorithm used in this paper.}
We use TT-DMRG-greedy, a cross-interpolation method, because the tensors are not
available in full and their ranks are unknown.
It requires $\mathcal{O}(dnr^2)$ tensor-entry evaluations and
$\mathcal{O}(dnr^3)$ additional algebraic operations
\citep{DolgovSavostyanov2019}, where
$d$ denotes the number of dimensions, $n$ a representative coordinate size,
and $r$ a representative TT rank. Hence, for bounded $n$ and $r$,
the complexity of the TT construction is linear in the dimension.
This estimate concerns the tensor algorithm itself; if the
cost of evaluating a single tensor entry grows with $d$, this cost
must be included separately in the overall computational complexity.

\subsection{Functional tensor trains (FTTs)}
\label{subsec:FTT}

The continuous analogue of a TT, for functions in tensor-product Hilbert spaces
\citep{hackbusch2009new,hackbusch2019tensor}, is a FTT
\citep{bigoni2016spectral,gorodetsky2019continuous}.

A rank-$\mathbf r=(r_1,\ldots,r_{d-1})$ FTT of
$g\in L^2(\mathbf D)$ is
\begin{equation}
\label{eq:FTT-rank-r}
    g_{\mathbf r\text{-TT}}(\bx)
    =
    \sum_{\balpha}^{\mathbf r}
    \prod_{j=1}^d
    B_j(
        \alpha_{j-1};
        x_j;
        \alpha_j
    )
    =
    \mathbf B_1(x_1)\cdots\mathbf B_d(x_d),
\end{equation}
where $\alpha_0=\alpha_d=1$ and
\[
    B_j(
        \alpha_{j-1};
        \cdot;
        \alpha_j
    )
    \in L^2(D_j).
\]
Here
\[
    \mathbf B_j(x_j)
    :=
    \left[
        B_j(\alpha_{j-1};x_j;\alpha_j)
    \right]_{\alpha_{j-1},\alpha_j},
    \qquad
    \sum_{\balpha}^{\mathbf r}
    :=
    \sum_{\alpha_1=1}^{r_1}\cdots
    \sum_{\alpha_{d-1}=1}^{r_{d-1}}.
\]
A key property used in the later derivations is that contraction against a
separable function factorizes across dimensions.

\begin{theorem}[Separable contraction]
\label{thm:FTT-inner-product}
Let $g_{\mathbf r\text{-TT}}$ be given by
\eqref{eq:FTT-rank-r}, and let
\[
    v(\bx)
    =
    \prod_{j=1}^d v_j(x_j),
    \qquad
    v_j\in L^2(D_j).
\]
Then, whenever the integrals exist,
\begin{equation}
\label{eq:FTT-inner-product}
    \int_{\mathbf D}
    g_{\mathbf r\text{-TT}}(\bx)
    v(\bx)\,d\bx
    =
    \prod_{j=1}^d
    \left[
        \int_{D_j}
        \mathbf B_j(x_j)
        v_j(x_j)\,dx_j
    \right],
\end{equation}
where the right-hand side is an ordered
matrix product.
\end{theorem}

\begin{proof}
Each summand of \eqref{eq:FTT-rank-r} multiplied by
$v$ is a product of univariate factors
$B_j(\alpha_{j-1};\cdot;\alpha_j)v_j$, each in
$L^1(D_j)$ by Cauchy--Schwarz, so the integrand lies
in $L^1(\mathbf D)$ and Fubini's theorem applies to the
finitely many summands separately.
Integrating each summand coordinatewise and
reassembling the sum over $\balpha$ as an ordered
matrix product gives
\eqref{eq:FTT-inner-product}.
\end{proof}

The proof does not require $\mathbf D$ to be bounded, so the theorem holds on
any product $D_1\times\cdots\times D_d$ of possibly unbounded intervals;
Section~\ref{subsec:COS-TT-CHF} applies it to the frequency variable on $\R^d$.

\section{Tensorized COS methods using TT}
\label{sec:method}

We derive COS-TT first: it establishes the tensorized representation and
notation that both methods share.

\subsection{COS-TT: via TT decomposition of the COS coefficient tensor}
\label{subsec:COS-TT}
COS-TT compresses the COS coefficient tensor $\mathcal{A}$ of the
multidimensional COS approximation \eqref{eq:fhat}: the COS formula relates the
ch.f.\ to the coefficients, and the TT step then acts only on that finite array.

Suppose that the COS coefficient tensor $\mathcal{A}$ has a rank-$\mathbf{r}$ TT
approximation $\hat{\mathcal{A}}$:
\begin{equation}
    \label{eq:A-TT-discrete}
    \hat{\mathcal{A}}[k_1,\ldots,k_d]
    :=
    \hat{\mathbf{A}}_1(k_1)
    \hat{\mathbf{A}}_2(k_2)
    \cdots
    \hat{\mathbf{A}}_d(k_d),
\end{equation}
where
\[
    \hat{\mathbf{A}}_j(k_j)
    \in
    \R^{r_{j-1}\times r_j},
    \quad
    r_0=r_d=1.\]
$\hat{\mathbf{A}}_j(k_j)$ is the $k_j$-th matrix slice of the $j$-th TT core,
and $\mathbf{r}=(r_1,\ldots,r_{d-1})$ contains the TT ranks.
Substituting this into the COS expansion and using
$\Phi_{\bk}(\bx)=\prod_j\phi^{(j)}_{k_j}(x_j)$ of \eqref{eq:product-basis} gives
\begin{equation}
    \hat{f}(\bx)
\approx \hat{f}_{\mathbf{r}\text{-CTT}}(\bx)
:=
    \sum_{\bk}^{\bK}
    \left(
        \prod_{j=1}^{d}
        \hat{\mathbf{A}}_j(k_j)
    \right)
    \prod_{j=1}^{d}
    \phi^{(j)}_{k_j}(x_j)
    =
    \prod_{j=1}^{d}
    \left(
        \sum_{k_j=0}^{K_j-1}
        \hat{\mathbf{A}}_j(k_j)
        \phi^{(j)}_{k_j}(x_j)
    \right),
    \label{eq:COS-TT-discrete}
\end{equation}
where each summation index occurs in only one matrix factor of the ordered
product, and the scalar basis functions commute with the matrices.
Equivalently,
\begin{equation}
    \label{eq:discrete-core-function}
    \hat{f}_{\mathbf{r}\text{-CTT}}(\bx)
    =\prod_{j=1}^{d} \hat{\bm{\Gamma}}^{C}_j(x_j),
    \quad\text{with}\quad
    \hat{\bm{\Gamma}}^{C}_j(x_j)
    :=
    \sum_{k_j=0}^{K_j-1}
    \hat{\mathbf{A}}_j(k_j)
    \phi^{(j)}_{k_j}(x_j)
    \in
    \R^{r_{j-1}\times r_j}.
\end{equation}
A TT decomposition therefore expresses the COS approximation as a chain of
univariate matrix-valued cosine expansions $\hat{\bm{\Gamma}}^{C}_j(x_j)$.
The superscript $C$ indicates that the compressed object is the COS coefficient
tensor and distinguishes these cores from the ch.f.-based cores introduced in
Section~\ref{subsec:COS-TT-CHF}. Equation~\eqref{eq:discrete-core-function} is
an FTT representation of the COS-TT density
$\hat f_{\mathbf r\text{-CTT}}$. Section~\ref{subsec:duality} discusses this
connection.

For the subsequent derivations, we display the auxiliary indices of each matrix
slice, writing
$[\hat{\mathbf{A}}_j(k_j)]_{\alpha_{j-1},\alpha_j}
= \hat{A}_j(\alpha_{j-1},k_j,\alpha_j)$, so that
$\hat{A}_j\in\R^{r_{j-1}\times K_j\times r_j}$. The matrix products of
\eqref{eq:A-TT-discrete} become sums over the auxiliary indices:
\begin{equation}
    \label{eq:A-TT}
    \hat{\mathcal{A}}_{\bk}
    =
    \sum_{\balpha}^{\mathbf r}
    \prod_{j=1}^{d}
    \hat{A}_j(
        \alpha_{j-1},
        k_j,
        \alpha_j
    ),
    \quad\text{with}\quad
    \alpha_0=\alpha_d=1,
\end{equation}
where the auxiliary-index summation is defined in Section~\ref{subsec:FTT}.
The same convention applies to the functional cores,
$[\hat{\bm{\Gamma}}^{C}_j(x_j)]_{\alpha_{j-1},\alpha_j}=
\hat{\Gamma}^{C}_j(\alpha_{j-1};x_j;\alpha_j)$. The corresponding COS-TT
approximation is \eqref{eq:COS-TT-discrete} in entry notation:
\begin{equation}
    \label{eq:COS-TT-approx}
    \boxed{
    \hat{f}_{\mathbf{r}\text{-CTT}}(\bx)
    :=
    \sum_{\balpha}^{\mathbf r}
    \prod_{j=1}^{d}
    \hat{\Gamma}^{C}_j(
        \alpha_{j-1};
        x_j;
        \alpha_j
    ),
    }
\end{equation}
with
\begin{equation}
    \label{eq:functional-core}
    \hat{\Gamma}^{C}_j(
        \alpha_{j-1};
        x_j;
        \alpha_j
    )
    :=
    \sum_{k_j=0}^{K_j-1}
    \hat{A}_j(
        \alpha_{j-1},
        k_j,
        \alpha_j
    )
    \phi^{(j)}_{k_j}(x_j).
\end{equation}

\paragraph{Summary.} The COS-TT method is the COS approximation
\eqref{eq:fhat} with $\mathcal{A}$ replaced by its TT approximation
$\hat{\mathcal{A}}$, written in low-rank separable form.
The approximation chain from $f$ to the COS-TT density
$\hat{f}_{\mathbf{r}\text{-CTT}}$ has three components:
\begin{equation}
    \label{eq:approx-chain}
    f
    \;\xrightarrow[\;\varepsilon_{\bK}\;]{\;\text{Fourier--cosine truncation}\;}
    \tilde{f}
    \;\xrightarrow[\;\varepsilon_{\mathrm{COS}}\;]{\;\mathcal{F}\to\mathcal{A}\;}
    \hat{f}
    \;\xrightarrow[\;\varepsilon_{\CTT}\;]{\;\text{TT decomposition}\;}
    \hat{f}_{\mathbf{r}\text{-CTT}},
\end{equation}
on the truncation hyperrectangle $\mathbf{D}$ defined in
\eqref{eq:domain-D}.
The three arrows indicate three error components, which add through the triangle
inequality. Appendix~\ref{sec:error} quotes bounds on the first two from the COS
literature and assumes a tolerance on the third.

Once the TT cores of $\mathcal{A}$ are available, density recovery has complexity
$\mathcal{O}(d\Kmax r^2)$, with $\Kmax=\max_j K_j$ and $r=\max_j r_j$.
TT-cross avoids forming the full tensor, but each requested entry of
$\mathcal{A}$ still costs $\mathcal{O}(2^{d-1})$ ch.f.\ evaluations. The offline
stage therefore remains subject to the curse of dimensionality.

\subsection{COS-TT-CHF: via TT decomposition of the ch.f.\ sample tensor}
\label{subsec:COS-TT-CHF}

Instead of approximating the Fourier--cosine coefficients through the COS
coefficient tensor $\mathcal A$, COS-TT-CHF uses the exact relation between the
ch.f.\ and the original coefficients derived below. We call this the
\emph{alternative COS representation}.

\subsubsection{The alternative COS representation}
We first derive the relation in one dimension. Let
$f, \varphi\in L^1(\R)$ form a continuous Fourier-transform pair:
\[
    \varphi(\omega)
    =
    \int_{\R}
    f(x)e^{i\omega x}\,dx
    \quad \textrm{and} \quad     
    f(x)
    =
    \frac{1}{2\pi}
    \int_{\R}
    \varphi(\omega)e^{-i\omega x}\,d\omega .
\]
Let $\phi_k$ be the normalized cosine basis of \eqref{eq:product-basis} at
$d=1$, which absorbs the primed-sum convention of
\eqref{eq:cosine-series-1d}. The corresponding coefficients and truncated
series are
\[
    \mathcal F_k
    =
    \int_a^b f(x)\phi_k(x)\,dx,
    \qquad
    \tilde f=\sum_{k=0}^{K-1}\mathcal F_k\phi_k.
\]
This is the $d=1$ case of \eqref{eq:multD_Fourier_cosine}, and
$\mathcal F_k=\beta_k(b-a)F_k/2$ for the $F_k$ of
\eqref{eq:cosine-series-1d}.
Substituting the inverse Fourier representation of $f$ and applying Fubini's
theorem gives the exact relation between $\varphi$ and $\mathcal F_k$:
\begin{equation}
    \label{eq:chf-side-1d}
    \boxed{
    \mathcal F_k
    =
    \frac{1}{2\pi}
    \int_{\R}
    \varphi(\omega)
    \Psi(\omega,k)\,d\omega,
    \quad \textrm{with} \quad
    \Psi(\omega,k)
    :=
    \int_a^b
    e^{-i\omega x}\phi_k(x)\,dx .
    }
\end{equation}
The kernel $\Psi(\cdot,k)$ is the truncated Fourier transform of $\phi_k$, and
it is bounded. Writing the cosine as a half-sum of exponentials gives a closed
form in terms of
\begin{equation}
    \label{eq:E-kernel}
    E(\nu;l,u)
    :=
    \int_{l}^{u}
    e^{i\nu x}\,dx
    =
    \frac{
        e^{i\nu u}-e^{i\nu l}
    }{i\nu}
    \quad \textrm{for} \;
    \nu\neq0,
    \quad \textrm{with} \quad
    E(0;l,u)
    =
    u-l .
\end{equation}

The relation extends directly to multiple dimensions. Let
$f(\bx),\varphi(\bomega)\in L^1(\R^d)$ form a Fourier-transform pair.
Then the Fourier--cosine coefficient tensor $\mathcal{F}_{\bk}$ of
\eqref{eq:F-tensor} and the ch.f.\ $\varphi(\bomega)$ satisfy
\begin{equation}
    \label{eq:chf-side}
    \mathcal{F}_{\bk}
    =
    \frac{1}{(2\pi)^d}
    \int_{\R^d}
    \varphi(\bomega)
    \bm{\Psi}_{\bk}(\bomega)\,d\bomega,
    \quad \textrm{with} \quad
    \bm{\Psi}_{\bk}(\bomega)
    :=
    \int_{\mathbf{D}}
    e^{-i\bomega\cdot\bx}
    \Phi_{\bk}(\bx)\,d\bx.
\end{equation}
The key property of the kernel is separability. Both $\Phi_{\bk}$ of
\eqref{eq:product-basis} and $e^{-i\bomega\cdot\bx}$ are products of univariate
factors, so the multidimensional kernel factorizes as
\begin{equation}
    \label{eq:Psi-separable}
    \bm{\Psi}_{\bk}(\bomega)
    =
    \prod_{j=1}^{d}
    \Psi_j(\omega_j,k_j),
    \quad \textrm{with} \quad    \Psi_j(\omega,k)
    :=
    \int_{l_j}^{u_j}
    e^{-i\omega x}
    \phi^{(j)}_k(x)\,dx.
\end{equation}
Each one-dimensional factor is available analytically, through the same
$E$ of \eqref{eq:E-kernel} evaluated on the $j$-th interval:
\begin{equation}
    \label{eq:Psi-closed-form}
    \Psi_j(\omega,k)
    =
    \frac{
        \beta^{(j)}_k
    }{2}
    \sum_{s=\pm1}
    \exp\!\left(
        -\frac{isk\pi l_j}{u_j-l_j}
    \right)
    E\!\left(
        \frac{sk\pi}{u_j-l_j}-\omega
        ;\,
        l_j,u_j
    \right).
\end{equation}
No ch.f.\ or model parameter appears in
\eqref{eq:Psi-closed-form}. The kernel can therefore be tabulated independently
of the probabilistic model.

Equations~\eqref{eq:chf-side} and \eqref{eq:Psi-separable} now yield the
core-wise TT relation between the ch.f.\ and $\mathcal{F}$.

\subsubsection{Core-wise relation between the ch.f.\
and the Fourier--cosine coefficient tensor}
\label{subsec:chf-core-relation}

The separability in \eqref{eq:Psi-separable} allows the alternative COS
representation \eqref{eq:chf-side} to be evaluated directly in FTT form.
Suppose that the ch.f.\ admits the continuous rank-$\mathbf{r}$ FTT
approximation
\begin{equation}
\label{eq:chf-FTT}
    \varphi(\bomega)
    \approx
    \varphi^{\mathbf{r}}(\bomega)
    :=
    \sum_{\balpha}^{\mathbf r}
    \prod_{j=1}^{d}
    Z_j(
        \alpha_{j-1};
        \omega_j;
        \alpha_j
    ),
\end{equation}
and assume that the following integrals exist.

Substituting \eqref{eq:chf-FTT} and
\eqref{eq:Psi-separable} into
\eqref{eq:chf-side}, and applying the separable
contraction property of
Theorem~\ref{thm:FTT-inner-product}, gives
\begin{equation}
    \label{eq:chf-F-TT}
    \mathcal{F}_{\bk}
    \approx
    \mathcal{F}^{\varphi}_{\mathbf{r},\bk}
    :=
    \sum_{\balpha}^{\mathbf r}
    \prod_{j=1}^{d}
    F^{\varphi}_j(
        \alpha_{j-1},
        k_j,
        \alpha_j
    ),
\end{equation}
with
\begin{equation}
    \label{eq:chf-cores}
    \boxed{
    F^{\varphi}_j(
        \alpha_{j-1},
        k_j,
        \alpha_j
    )
    =
    \frac{1}{2\pi}
    \int_{\R}
    Z_j(
        \alpha_{j-1};
        \omega;
        \alpha_j
    )
    \Psi_j(
        \omega,
        k_j
    )
    \,d\omega.
    }
\end{equation}
The kernel factor $\Psi_j$ is known analytically from
\eqref{eq:Psi-closed-form}, so \eqref{eq:chf-cores}
is computable once the cores $Z_j$ are available.

Thus, the transformation from the ch.f.\ FTT to the Fourier--cosine
coefficient TT acts independently in each dimension: the $j$-th coefficient
core depends only on the $j$-th ch.f.\ core. Equation~\eqref{eq:chf-cores}
integrates only the frequency variable and leaves the auxiliary indices
unchanged, so it does not increase the TT ranks.
The approximation in \eqref{eq:chf-F-TT} is inherited
solely from the FTT approximation
\eqref{eq:chf-FTT}; the relation
\eqref{eq:chf-side} itself is exact.

Inserting \eqref{eq:chf-F-TT} into \eqref{eq:multD_Fourier_cosine} and
regrouping the $\bk$ sum by dimension, as in \eqref{eq:COS-TT-discrete}, gives a
density approximation of the same form as
\eqref{eq:COS-TT-approx}--\eqref{eq:functional-core}, with the coefficient cores
$\hat{A}_j$ of $\mathcal A$ replaced by the cores $F^{\varphi}_j$ of
$\mathcal F$; the superscript $\varphi$ identifies the compressed object.
Synthesizing those cores in the cosine basis gives the density FTT characterized
in Section~\ref{subsec:duality}. What it approximates is $\tilde f$ rather than
$\hat f$, because \eqref{eq:chf-side} is exact: on this route the only
approximation so far is the FTT approximation \eqref{eq:chf-FTT}.

The cores $Z_j$ are still idealized. In practice we sample the ch.f.\ on a
finite frequency grid and compute discrete cores. The derivation below repeats
the steps above with frequency-domain truncation and quadrature in place, and
displays the resulting synthesis once, in
\eqref{eq:COS-TT-CHF-approx-quad}--\eqref{eq:functional-core-chf-quad}.

\subsubsection{COS-TT-CHF: the same derivation in discretized form}
\label{subsec:chf-cores-computation}

We first truncate the frequency domain to a finite
hyperrectangle:
\begin{equation}
    \label{eq:domain-Omega}
    \bm{\Omega}
    :=
    [-\Omega_1,\Omega_1]
    \times\cdots\times
    [-\Omega_d,\Omega_d]
    \subset\R^d,
    \qquad
    \Omega_j>0,
\end{equation}
with $\Omega_j$ the half-width in coordinate $j$. We also write $\bm{\Omega}$
for the vector $(\Omega_1,\ldots,\Omega_d)^\top$ of half-widths.

Define
\begin{equation}
    \label{eq:F-Omega}
    \mathcal{F}_{\bk}\approx
    \tilde{\mathcal{F}}_{\bk}
    :=
    \frac{1}{(2\pi)^d}
    \int_{\bm{\Omega}}
    \varphi(\bomega)
    \bm{\Psi}_{\bk}(\bomega)\,d\bomega.
\end{equation}
Because the frequency-domain range is a product of one-dimensional intervals, the
core-wise factorization remains valid after truncation. The integrals in
\eqref{eq:chf-cores} are simply restricted from $\R$ to
$[-\Omega_j,\Omega_j]$.

Then we apply a quadrature rule to 
\eqref{eq:F-Omega}.
In coordinate $j$, let
\[
    \omega^{(j)}_1,
    \ldots,
    \omega^{(j)}_{M_j}
    \in
    [-\Omega_j,\Omega_j]
\]
be quadrature points with corresponding weights
$w^{(j)}_1,\ldots,w^{(j)}_{M_j}$, and call
$M_j$ the quadrature point count in coordinate $j$.
The sampled ch.f.\ tensor is
\begin{equation}
    \label{eq:chf-sampled}
    \bm{\varphi}[m_1,\ldots,m_d]
    :=
    \varphi\!\left(
        \omega^{(1)}_{m_1},
        \ldots,
        \omega^{(d)}_{m_d}
    \right).
\end{equation}
Applying the quadrature rule to \eqref{eq:F-Omega} with the exact samples
\eqref{eq:chf-sampled} defines the discretized coefficient tensor
\begin{equation}
    \label{eq:F-Omega-M}
    \tilde{\mathcal{F}}_{\bk}\approx
    \hat{\mathcal{F}}_{\bk}
    :=
    \frac{1}{(2\pi)^d}
    \sum_{m_1=1}^{M_1}
    \cdots
    \sum_{m_d=1}^{M_d}
    \left(
        \prod_{j=1}^{d} w^{(j)}_{m_j}
    \right)
    \bm{\varphi}[m_1,\ldots,m_d]\,
    \bm{\Psi}_{\bk}\!\left(
        \omega^{(1)}_{m_1},\ldots,\omega^{(d)}_{m_d}
    \right).
\end{equation}

Equation~\eqref{eq:F-Omega-M} defines the discretized coefficient tensor but is
not evaluated as a full tensor-product sum. Instead, cross interpolation
approximates the sampled ch.f.\ tensor by a discrete TT:
\begin{equation}
    \label{eq:chf-sampled-TT}
    \bm{\varphi}
    \approx
    \hat{\bm{\varphi}}, \qquad
    \hat{\bm{\varphi}}[m_1,\ldots,m_d]
    :=
    \sum_{\balpha}^{\mathbf r}
    \prod_{j=1}^{d}
    \hat{Z}_j(
        \alpha_{j-1},
        m_j,
        \alpha_j
    ),
\end{equation}
where the hat indicates that the computed TT cores are approximate.
Substituting \eqref{eq:chf-sampled-TT} into \eqref{eq:F-Omega-M} and using the
kernel factorization \eqref{eq:Psi-separable} gives the core-wise transformation
\begin{equation}
    \label{eq:chf-cores-quad}
    \boxed{
    \hat{F}^{\varphi}_j(
        \alpha_{j-1},
        k_j,
        \alpha_j
    )
    =
    \frac{1}{2\pi}
    \sum_{m=1}^{M_j}
    w^{(j)}_m
    \hat{Z}_j(
        \alpha_{j-1},
        m,
        \alpha_j
    )
    \Psi_j\!\left(
        \omega^{(j)}_m,
        k_j
    \right).
    }
\end{equation}
That is, \eqref{eq:chf-cores-quad} discretizes the exact core-wise relation
\eqref{eq:chf-cores}: the frequency integral is replaced
by its quadrature sum, and the functional ch.f.\ cores $Z_j$ by the discrete
cores $\hat{Z}_j$ of the sampled tensor.
For fixed core dimensions this is a matrix multiplication between the sampled
ch.f.\ core and the precomputed kernel matrix, so the non-increase of TT ranks
noted above carries over. Let $\hat{\mathcal{F}}^{\varphi}_{\mathbf{r}}$ denote
the coefficient tensor represented by \eqref{eq:chf-cores-quad}, equivalently
\eqref{eq:F-Omega-M} with $\bm{\varphi}$ replaced by
$\hat{\bm{\varphi}}$. Let $\hat{f}_{\mathbf{r}\text{-}\varphi\text{TT}}$ denote
the density it synthesizes. These are the COS-TT-CHF counterparts of
$\hat{\mathcal{A}}$ and $\hat{f}_{\mathbf{r}\text{-CTT}}$.

TT-DMRG-greedy computes the cores in \eqref{eq:chf-sampled-TT} from selected
entries. Each costs one ch.f.\ evaluation rather than the $2^{d-1}$ evaluations
required for a COS coefficient. The tensor-algorithm complexity of
Section~\ref{subsec:TT-algorithms} therefore applies separately from the
model-dependent cost of one ch.f.\ evaluation.

In coefficient form, \eqref{eq:chf-F-TT} now uses the computed cores
$\hat{F}^{\varphi}_j$ in place of the idealized $F^{\varphi}_j$, so that
$\mathcal{F}^{\varphi}_{\mathbf{r}}$ becomes
$\hat{\mathcal{F}}^{\varphi}_{\mathbf{r}}$. Synthesizing these cores in the
cosine basis gives the COS-TT-CHF density:
\begin{equation}
    \label{eq:COS-TT-CHF-approx-quad}
    \boxed{
    \tilde{f}(\bx)
    \approx
    \hat{f}_{\mathbf{r}\text{-}\varphi\text{TT}}(\bx)
    :=
    \sum_{\balpha}^{\mathbf r}
    \prod_{j=1}^{d}
    \hat{\Gamma}^{\varphi}_j(
        \alpha_{j-1};
        x_j;
        \alpha_j
    ),
    }
\end{equation}
with 
\begin{equation}
    \label{eq:functional-core-chf-quad}
    \boxed{
    \hat{\Gamma}^{\varphi}_j(
        \alpha_{j-1};
        x_j;
        \alpha_j
    )
    :=
    \sum_{k_j=0}^{K_j-1}
    \hat{F}^{\varphi}_j(
        \alpha_{j-1},
        k_j,
        \alpha_j
    )
    \phi^{(j)}_{k_j}(x_j).}
\end{equation}

\paragraph{Summary.} The COS-TT-CHF approximation chain separates
Fourier--cosine truncation from the ch.f.-side approximation of $\mathcal F$.
The latter comprises three operations:
\begin{equation}
    \label{eq:chf-chain}
    \begin{aligned}
    f
    \;&
    \xrightarrow[\;\varepsilon_{\bK}\;]{\;\text{Fourier--cosine truncation}\;}
    \;
    \tilde{f}
    \;
    \xrightarrow[\;
        \varepsilon_{\bm{\Omega}},\,
        \varepsilon_{\mathrm{quad}},\,
        \varepsilon_{\varphi\TT}
    \;]{\;
        \mathcal{F}
        \,\to\,
        \hat{\mathcal{F}}^{\varphi}_{\mathbf{r}}
    \;}
    \;
    \hat{f}_{\mathbf{r}\text{-}\varphi\text{TT}},
    \\[0.6em]
    \text{with}
    \;&
    \underbrace{
        \mathcal{F}
    }_{\text{exact, \eqref{eq:chf-side}}}
    \;
    \xrightarrow[\;\varepsilon_{\bm{\Omega}}\;]{
        \;
        \int_{\R^d}
        \,\to\,
        \int_{\bm{\Omega}}
        \;
    }
    \;
    \underbrace{
        \tilde{\mathcal{F}}
    }_{\text{\eqref{eq:F-Omega}}}
    \;
    \xrightarrow[\;\varepsilon_{\mathrm{quad}}\;]{
        \;
        \int
        \,\to\,
        \textstyle\sum_{\mathbf m}\prod_j w^{(j)}_{m_j}
        \;
    }
    \;
    \underbrace{
        \hat{\mathcal{F}}
    }_{\text{\eqref{eq:F-Omega-M}}}
    \;
    \xrightarrow[\;\varepsilon_{\varphi\TT}\;]{
        \;
        \bm{\varphi}
        \,\to\,
        \hat{\bm{\varphi}}
        \;
    }
    \;
    \underbrace{
        \hat{\mathcal{F}}^{\varphi}_{\mathbf{r}}
    }_{\text{\eqref{eq:chf-cores-quad}}}.
    \end{aligned}
\end{equation}
These three ch.f.-side steps are specific to COS-TT-CHF and replace the
single COS-TT coefficient approximation $\mathcal{F}\to\mathcal{A}$.
Section~\ref{sec:error-chf} analyses them in this order and derives from the
first two the selection rules for $\bm{\Omega}$ and $\mathbf{M}$.

\subsection{Relation to functional tensor trains}
\label{subsec:duality}

Both methods construct cosine-basis FTTs from the ch.f.\ without evaluating the
density. Let $g_{\mathbf r\text{-TT}}$ be the FTT \eqref{eq:FTT-rank-r} with
functional cores $B_j$, and define their Fourier--cosine coefficient cores by
$G_j(\alpha_{j-1},k_j,\alpha_j)
:=\langle B_j(\alpha_{j-1};\cdot;\alpha_j),
\phi^{(j)}_{k_j}\rangle_{L^2(D_j)}$.
By separability of the basis \eqref{eq:product-basis},
Theorem~\ref{thm:FTT-inner-product} with $v=\Phi_{\bk}$ gives
\[
\langle g_{\mathbf r\text{-TT}},\Phi_{\bk}\rangle_{L^2(\mathbf D)}
=\sum_{\balpha}^{\mathbf r}\prod_{j=1}^{d}
G_j(\alpha_{j-1},k_j,\alpha_j).
\]
Thus, for every finite $\bK$, the coefficient tensor is a discrete TT
\eqref{eq:TT-discrete} with ranks bounded by $\mathbf r$. Conversely, synthesizing these cores
in the cosine basis, as in \eqref{eq:discrete-core-function} and
\eqref{eq:functional-core-chf-quad}, gives the orthogonal projection of the FTT,
which converges to it in $L^2$ as $\bK\to\infty$. Analysis and synthesis act
only on the physical index and leave the auxiliary indices unchanged. Hence the
complete coefficient and functional representations are equivalent. The
spectral FTT literature
\citep{bigoni2016spectral,gorodetsky2019continuous} compresses the function
before expanding its cores and therefore requires pointwise evaluations of $f$.
The tensorized COS methods instead construct the coefficient cores from the
ch.f.

At the representation level, this identification also establishes convergence
under a weaker assumption than the existing spectral derivation. For a
fixed-rank FTT, the $\bK$-limit follows from completeness of the cosine basis.
For the rank limit, $f\in L^2(\mathbf D)$ makes every unfolding of $f$ a
Hilbert--Schmidt operator, so iterated Schmidt decomposition yields finite-rank
FTTs converging to $f$ in $L^2(\mathbf D)$. Moreover, orthogonality of
$\{\Phi_{\bk}\}$ makes the Fourier--cosine projection error orthogonal to the
subsequent fixed-rank coefficient error, so the two errors add in squares. By
contrast, the polynomial spectral construction assumes
H\"older continuity with exponent above $1/2$ for the rank limit
\citep[Thm.~13]{bigoni2016spectral} and membership of
$\mathcal H^{s}(\mathbf D)$ with $s>1$ for the projection error
\citep[Prop.~3]{bigoni2016spectral}, combining them through the triangle
inequality in their eq.~(66). Both hypotheses are needed because its core
truncation is not an orthogonal projection. The weaker assumption guarantees
convergence but not a rate. Rates require coefficient-decay assumptions, stated
separately in Appendix~\ref{sec:error}. The performance of rank-adaptive
TT-cross remains an algorithmic question covered by
Assumption~\ref{ass:training}.

\section{Error analysis}
\label{sec:error-chf}
\label{subsec:chf-error-sources}

This section analyses the errors specific to the ch.f.\ route. Two of the terms
below fall outside it: the COS coefficient error $\varepsilon_{\mathrm{COS}}$,
which COS-TT-CHF avoids because \eqref{eq:chf-side} is exact, and the
series-truncation error $\varepsilon_{\bK}$, which it shares with COS-TT. Both
are covered by standard COS theory
\citep{fang2009novel,junike2025characteristic} and Appendix~\ref{sec:error}.

We measure the three ch.f.-side errors on the coefficient side. By
Theorem~\ref{thm:train-conversion}, the Frobenius error equals the $L^2$ error
of the synthesized density, so we do not repeat this conversion below.

The chain \eqref{eq:chf-chain} and the triangle inequality then give the
counterpart of \eqref{eq:error-decomp} for COS-TT-CHF:
\begin{equation}
    \label{eq:error-decomp-chf}
    \norm{
        f - \hat{f}_{\mathbf{r}\text{-}\varphi\text{TT}}
    }_{L^2(\mathbf D)}
    \leq
    \underbrace{
        \norm{f-\tilde f}_{L^2(\mathbf D)}
    }_{\varepsilon_{\bK}}
    +
    \underbrace{
        \norm{\mathcal F-\tilde{\mathcal F}}_F
    }_{\varepsilon_{\bm{\Omega}}}
    +
    \underbrace{
        \norm{\tilde{\mathcal F}-\hat{\mathcal F}}_F
    }_{\varepsilon_{\mathrm{quad}}}
    +
    \underbrace{
        \norm{
            \hat{\mathcal F}
            -
            \hat{\mathcal F}^{\varphi}_{\mathbf r}
        }_F
    }_{\varepsilon_{\varphi\TT}} .
\end{equation}
The remaining three terms are taken below in the order of the chain; the first
two also yield the selection rules for $\bm{\Omega}$ and $\mathbf M$.

\subsection{Frequency-domain truncation error
\texorpdfstring{$\varepsilon_{\bm{\Omega}}$}{eps Omega}}
\label{subsec:chf-new-errors}

Kernel localization couples the frequency-domain truncation range
to $\mathbf D$ and $\bK$.
With $|D_j|=u_j-l_j$ as above, let
\begin{equation}
    \label{eq:kappa-definition}
    \kappa_j
    :=
    \frac{(K_j-1)\pi}{|D_j|}.
\end{equation}
By \eqref{eq:Psi-closed-form},
$\Psi_j(\cdot,k_j)$ is centred at
$\pm k_j\pi/|D_j|$.
Hence $[-\Omega_j,\Omega_j]$
contains the centres of all retained cosine terms if and only if
\begin{equation}
    \label{eq:Omega-compatibility}
    \Omega_j
    \geq
    \kappa_j,
    \qquad
    j=1,\ldots,d .
\end{equation}
Condition \eqref{eq:Omega-compatibility}
is thus a lower bound imposed by the cosine expansion.
It ensures coverage of the retained cosine frequencies but not accuracy,
which also depends on the omitted ch.f.\ tail.
The following proposition quantifies this requirement and gives a
tolerance-dependent selection rule for $\Omega_j$ when ch.f.\ tail constants are
available.

\begin{proposition}
[Frequency-domain truncation error and selection]
\label{prop:Omega-selection}

Let $\tau_{\bm{\Omega}}>0$
be a prescribed tolerance for
frequency-domain truncation.
Assume that, for each $j$ and all $s>0$,
\begin{equation}
    \label{eq:chf-tail-decay}
    \left(
        \int_{\{|\omega_j|>s\}}
        |\varphi(\bomega)|^2
        \,d\bomega
    \right)^{1/2}
    \leq
    C_j s^{-q_j},
    \qquad
    q_j>0 .
\end{equation}
Then, with $\kappa_j$ defined by \eqref{eq:kappa-definition}, the condition
\begin{equation}
    \label{eq:Omega-selection}
    \boxed{
    \Omega_j
    \geq
    \max
    \left\{
        \kappa_j,
        \left(
            \frac{
                \sqrt d\,C_j
            }{
                (2\pi)^{d/2}
                \tau_{\bm{\Omega}}
            }
        \right)^{1/q_j}
    \right\}
    }
\end{equation}
ensures both coverage of the retained cosine frequencies and
\(
    \left\|
        \mathcal F-\tilde{\mathcal F}
    \right\|_F
    \leq
    \tau_{\bm{\Omega}} .
\)
\end{proposition}

\begin{proof}

Since $|\varphi|\leq1$ on the bounded region $\bm{\Omega}$ and
\eqref{eq:chf-tail-decay} makes the tail integral finite, $\varphi\in
L^2(\R^d)$ and hence $f\in L^2(\R^d)$, which is what the Plancherel step below
requires. This is slightly more than the $L^2(\mathbf D)$ assumption of
Remark~\ref{rem:norms}.

Define the frequency-truncated inverse
Fourier transform by
\begin{equation}
    \label{eq:density-Omega}
    f_{\bm{\Omega}}(\bx)
    :=
    \frac{1}{(2\pi)^d}
    \int_{\bm{\Omega}}
    \varphi(\bomega)
    e^{-i\bomega\cdot\bx}
    \,d\bomega .
\end{equation}
Equivalently,
$f_{\bm{\Omega}}$ is the inverse Fourier
transform of
\(
    \varphi(\bomega)
    \mathbf 1_{\bm{\Omega}}(\bomega).
\)

Interchanging the order of integration in \eqref{eq:F-Omega}, as in the
derivation of \eqref{eq:chf-side}, gives
$\tilde{\mathcal F}_{\bk}=\int_{\mathbf D}f_{\bm{\Omega}}\Phi_{\bk}\,d\bx$.
Since $\mathcal F_{\bk}=\int_{\mathbf D}f\Phi_{\bk}\,d\bx$ by
\eqref{eq:F-tensor}, the two differ by
$\mathcal F_{\bk}-\tilde{\mathcal F}_{\bk}
=\langle f-f_{\bm{\Omega}},\Phi_{\bk}\rangle_{L^2(\mathbf D)}$.
Since the retained
$\{\Phi_{\bk}\}$ are orthonormal,
Bessel's inequality gives
\begin{equation}
    \label{eq:Bessel-F-error}
    \|
        \mathcal F
        -
        \tilde{\mathcal F}
    \|_F^2
    =
    \sum_{\bk}
    \left|
        \left\langle
            f-f_{\bm{\Omega}},
            \Phi_{\bk}
        \right\rangle
    \right|^2
    \leq
    \|f-f_{\bm{\Omega}}\|_{L^2(\mathbf D)}^2 .
\end{equation}
No basis constant appears, because the normalization
$\beta^{(j)}_{k_j}$ of \eqref{eq:product-basis} makes
$\|\Phi_{\bk}\|_{L^2(\mathbf D)}=1$ for every retained
$\bk$, the index $k_j=0$ included. With the unnormalized cosines of
\eqref{eq:cosine-series-1d} the same argument would carry a factor
$\mathrm{vol}(\mathbf D)^{1/2}$ into \eqref{eq:Omega-selection}, growing with
the physical-domain truncation range.

From \eqref{eq:density-Omega},
the Fourier transform of
$f-f_{\bm{\Omega}}$ is
\(
    \varphi(\bomega)
    \mathbf 1_{
        \R^d\setminus\bm{\Omega}
    }(\bomega).
\)
Under the Fourier convention used here,
Plancherel's theorem therefore yields
\begin{align}
&
\|f-f_{\bm{\Omega}}\|_{L^2(\R^d)}
=
\frac{1}{(2\pi)^{d/2}}
\left(
    \int_{\R^d\setminus\bm{\Omega}}
    |\varphi(\bomega)|^2
    \,d\bomega
\right)^{1/2}.
\label{eq:Plancherel-Omega}
\end{align}
Two elementary inclusions now finish the estimate. Since $\mathbf D\subset\R^d$,
the left-hand side of \eqref{eq:Bessel-F-error} is bounded by the global norm
\eqref{eq:Plancherel-Omega}. Since the complement of the product region
$\bm{\Omega}=\prod_{j=1}^d[-\Omega_j,\Omega_j]$ lies in the union of the $d$
slabs $\{|\omega_j|>\Omega_j\}$, the integral in \eqref{eq:Plancherel-Omega} is
bounded by the sum of the $d$ slab integrals. Together they give the
coordinate-wise bound
\begin{equation}
\label{eq:Omega-coordinate-bound}
\|
    \mathcal F-\tilde{\mathcal F}
\|_F
\leq
\frac{1}{(2\pi)^{d/2}}
\left(
    \sum_{j=1}^d
    \int_{\{|\omega_j|>\Omega_j\}}
    |\varphi(\bomega)|^2
    \,d\bomega
\right)^{1/2},
\end{equation}
which holds for any ch.f.\ tail. The hypothesis \eqref{eq:chf-tail-decay} bounds
the $j$-th slab integral by $C_j^2\Omega_j^{-2q_j}$. Allocating the tolerance
equally among the $d$ coordinates,
$(2\pi)^{-d/2}C_j\Omega_j^{-q_j}\leq\tau_{\bm{\Omega}}/\sqrt d$ for
$j=1,\ldots,d$, is then sufficient for
$\|\mathcal F-\tilde{\mathcal F}\|_F\leq\tau_{\bm{\Omega}}$, and solving for
$\Omega_j$ gives the second entry of the maximum in
\eqref{eq:Omega-selection}.

Finally, \eqref{eq:Omega-compatibility} imposes the lower bound on $\Omega_j$
needed to contain the centres of all retained kernel terms. Combining all
requirements gives \eqref{eq:Omega-selection}.
\end{proof}

When the tail constants in \eqref{eq:chf-tail-decay} are unavailable,
\eqref{eq:Omega-compatibility} provides a directly computable baseline. In
practice, one starts with $\Omega_j\geq\kappa_j$ and increases $\Omega_j$ until
the computed quantities stabilize. Section~\ref{subsec:num-window} tests this
baseline numerically.

\subsection{Quadrature error
\texorpdfstring{$\varepsilon_{\mathrm{quad}}$}{eps quad}}
\label{subsec:chf-difficulty}

We evaluate the truncated integral \eqref{eq:F-Omega} by
$M_j$-point Gauss--Legendre quadrature on $[-\Omega_j,\Omega_j]$. The integrand
is the product of the ch.f.\ and $\bm{\Psi}_{\bk}$, so the quadrature error
$\varepsilon_{\mathrm{quad}}$ depends on both the model and the cosine
parameters $\mathbf D$ and $\bK$.

Figure~\ref{fig:chf-oscillations} illustrates the resolution problem. The
integrand is smooth but can be strongly oscillatory, and its resolution depends
on both the quadrature point count and the location within the truncation
window. Smoothness alone therefore does not determine $\mathbf M$.

\begin{figure}[h]
  \centering
  \includegraphics[width=0.90\textwidth]{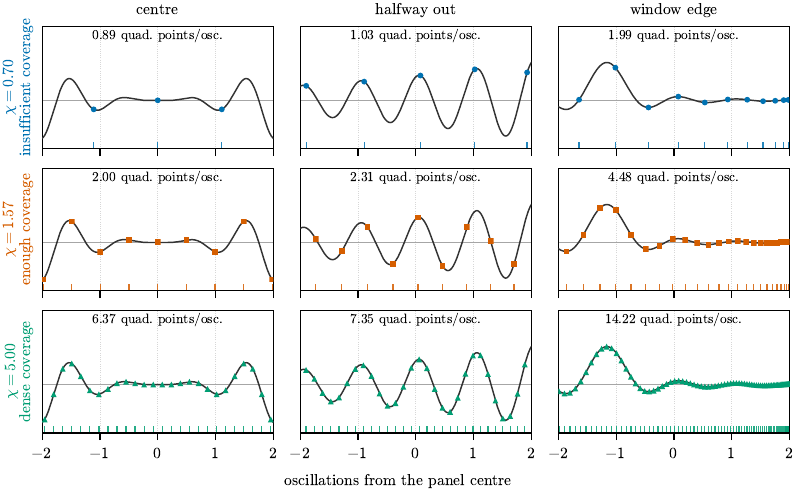}
  \caption{
  The same oscillatory kernel at three quadrature point counts $M$ (rows) and
  three window positions (columns); $K=64$, $\lambda_1=8$. Rows are labelled by
  the unsubscripted $\chi := M/(c(K-1))$ with $c := \Omega/\kappa$, which is a
  rescaling of the ratio $M/\chi_j$ and not the dimensionless frequency $\chi_j$
  of \eqref{eq:chi-def}: $\chi = (\pi/2)\,M/\chi_j$. The rows sit below, at, and
  well above
  $\chi = \pi/2$, equivalently $M = \chi_j$, where Gauss--Legendre places two
  quadrature points per oscillation at the window centre. Columns move the same
  oscillations from the centre of $\bm\Omega$ to its edge, where arcsine
  clustering gives the same $M$ more points per oscillation.}
  \label{fig:chf-oscillations}
\end{figure}

For smooth or analytic integrands,
Gauss quadrature converges rapidly
\citep{trefethen2008gauss,
xiang2012convergence}.
For highly oscillatory integrands, however,
this regime is reached only after the
oscillations are sufficiently resolved
\citep{deano2018computing}.
This motivates a phase-resolution scaling for
$\mathbf M$.

Consider a factor $e^{ir\omega_j}$, where $r\in\R$ is the phase rate and
$|r|$ is the oscillation rate in $\omega_j$.
After mapping
$\omega_j=\Omega_j t$ to $[-1,1]$,
the dimensionless frequency of the factor is
\begin{equation}
    \label{eq:chi-def}
    \chi_j
    :=
    \Omega_j|r|.
\end{equation}
For an integrand $h$, the classical $M_j$-point Gauss--Legendre
remainder is
$E_{M_j}=c_{M_j}h^{(2M_j)}(\xi)$
for some $\xi\in(-1,1)$, with
$c_{M_j}=2^{2M_j+1}(M_j!)^4/\{(2M_j+1)[(2M_j)!]^3\}$
\citep[Ch.~5]{atkinson1989introduction}.
All derivatives of order $2M_j$ of
$\cos(\chi_jt)$ and $\sin(\chi_jt)$
are bounded by $\chi_j^{2M_j}$, so, by Stirling's
formula,
\[
    |E_{M_j}|
    \leq
    c_{M_j}\chi_j^{2M_j}
    =
    O\!\left[
        M_j^{-1/2}
        \left(
            \frac{e\,\chi_j}{4M_j}
        \right)^{2M_j}
    \right],
\]
where $e=\exp(1)$ is Euler's number. The bound begins to decay only when
$M_j>(e/4)\chi_j\approx0.68\,\chi_j$, so $M_j$ must scale linearly with the
effective oscillation frequency before rapid convergence can be expected.
Section~\ref{subsec:num-window} observes the onset of the aliasing regime near
this predicted threshold.

For the present kernel, \eqref{eq:E-kernel} shows that
$E(q-\omega;l_j,u_j)$ carries the phases
$e^{-i\omega u_j}$ and $e^{-i\omega l_j}$, so its
phase rates in $\omega$ are $|l_j|$ and $|u_j|$.
Multiplication by the ch.f.\ may shift or increase these rates.
Write each oscillatory component of the integrand in \eqref{eq:chf-side} as
$A(\bomega)e^{i\theta(\bomega)}$, with $A$ slowly varying, and let
$R_j^{\mathrm{eff}} := \sup\bigl|\partial\theta/\partial\omega_j\bigr|$
denote the largest phase rate over these components.
This suggests the phase-resolution scaling
$M_j \gtrsim \Omega_j R_j^{\mathrm{eff}}$. We use the lower-bound notation
rather than $O(\cdot)$ because this scaling specifies a minimum quadrature
resolution.

A directly usable special case is
$\varphi(\bomega)
= e^{i\bmu\cdot\bomega}\varphi_0(\bomega)$,
as in many location-scale models,
where $\varphi_0$ is nonoscillatory.
Multiplication with the kernel then produces
the phases
$e^{i(\mu_j-u_j)\omega_j}$ and
$e^{i(\mu_j-l_j)\omega_j}$, so
$R_j^{\mathrm{eff}}
= \max\{|\mu_j-l_j|,|\mu_j-u_j|\}$.
If $D_j$ is centred at $\mu_j$, then
$R_j^{\mathrm{eff}}=|D_j|/2$, giving the practical resolution guide
\begin{equation}
    \label{eq:M-centered-guide}
    \boxed{
        M_j
        \gtrsim
        \frac{\Omega_j|D_j|}{2}
        .
    }
\end{equation}

This centered-domain guide also explains the coupling with $\bK$.
If $\Omega_j$ is chosen near the
compatibility scale of
\eqref{eq:Omega-compatibility},
that is
$\Omega_j=c_{\Omega,j}\kappa_j
=c_{\Omega,j}(K_j-1)\pi/|D_j|$
with a fixed multiplier $c_{\Omega,j}\geq1$, independent of $K_j$ and $|D_j|$,
then \eqref{eq:M-centered-guide} gives
$M_j\gtrsim K_j$.
Since \eqref{eq:Omega-compatibility} is a
lower bound, a wider window required by the
ch.f.\ tail increases $M_j$ proportionally.
Thus $\mathbf M$ is coupled to the
cosine expansion rather than determined
by the ch.f.\ alone.

Section~\ref{subsec:num-window} examines the resulting error. The ratio
$M_j/\chi_j$ of \eqref{eq:M-over-chi} governs the transition, and
Figure~\ref{fig:nodes} shows an error plateau near $M_j=\chi_j$ for GBM, at the
resolution that Figure~\ref{fig:chf-oscillations} reaches at the window centre.

The general phase-resolution scaling is heuristic, whereas
\eqref{eq:M-centered-guide} is directly usable in the centered location-scale
case above. Neither is an a priori tolerance bound.
In practice, $M_j$ is increased until
the computed quantities stabilize at
the required tolerance.
Increasing $\Omega_j$ reduces
$\varepsilon_{\bm{\Omega}}$ but also
increases the required quadrature
resolution. Thus, $\bm{\Omega}$ and
$\mathbf M$ should be selected jointly.

Oscillation constrains accuracy, but not necessarily TT rank:
$e^{i\bomega^\top\bmu}=\prod_j e^{i\omega_j\mu_j}$
is strongly oscillatory yet exactly separable, with TT rank one on any grid.
Too narrow a frequency-domain truncation range and too coarse a quadrature
therefore inflate $\varepsilon_{\bm{\Omega}}$ and $\varepsilon_{\mathrm{quad}}$
without inflating the ranks.

\begin{remark}[Oscillatory quadrature]
\label{rem:oscillatory-quadrature}

Standard Gauss--Legendre quadrature does
not exploit the known oscillatory structure
and requires increasing resolution as the effective
frequency grows
\citep{deano2018computing}.
Filon, Levin, numerical steepest-descent and oscillatory Gaussian quadratures
are designed for this regime. Because the oscillatory factors in $\Psi_j$ are
known analytically, applying such rules to the one-dimensional core
transformations may reduce the required $M_j$ without changing the TT
construction.
\end{remark}

\subsection{Propagated TT-decomposition error: 
\texorpdfstring{$\varepsilon_{\varphi\TT}$}{eps phi TT}}
\label{subsec:chf-tt-error}

The decomposition error itself arises in the same way in both methods, from rank
truncation and from the TT-cross approximation, and we do not analyse it here.
As in Appendix~\ref{sec:error}, we assume instead that the decomposition attains
a prescribed tolerance,
\begin{equation}
    \label{eq:delta-phiTT}
    \norm{
        \bm{\varphi}
        -
        \hat{\bm{\varphi}}
    }_F
    \leq
    \delta_{\varphi\TT},
\end{equation}
the counterpart for $\bm{\varphi}$ of Assumption~\ref{ass:training} for
$\mathcal A$. Section~\ref{subsec:eps-TT} explains why this quantity is assumed
rather than derived, and how it is verified numerically.

What differs between the two methods is what becomes of that tolerance
afterwards, and this is what the present subsection bounds. Only the compressed
object changes: COS-TT compresses $\mathcal A$, whereas COS-TT-CHF compresses
the sampled ch.f.\ tensor $\bm{\varphi}$. The entries of $\mathcal A$ are
already coefficients, so for COS-TT the tolerance is a coefficient error and
Theorem~\ref{thm:train-conversion} carries it straight to the density. The
entries of $\bm{\varphi}$ are not, so for COS-TT-CHF the tolerance must still
pass through the transformation \eqref{eq:chf-cores-quad} that turns ch.f.\
cores into coefficient cores. What is bounded below is therefore the image of
$\delta_{\varphi\TT}$ under that transformation, and the question is how much
the transformation can amplify it.

The dimension-wise transformation
\eqref{eq:chf-cores-quad} is linear.
With
\[
    \mathbf{Q}_j[k,m]
    :=
    \frac{1}{2\pi}
    w_m^{(j)}
    \Psi_j(
        \omega_m^{(j)},k
    ),
\]
we have
\[
\begin{aligned}
&
\operatorname{vec}
\left(
    \hat{\mathcal F}
    -
    \hat{\mathcal F}^{\varphi}_{\mathbf r}
\right)
=
\left(
    \mathbf{Q}_d\otimes\cdots\otimes \mathbf{Q}_1
\right)
\operatorname{vec}
\left(
    \bm{\varphi}
    -
    \hat{\bm{\varphi}}
\right).
\end{aligned}
\]
Hence
\begin{equation}
\label{eq:chf-TT-transfer}
\varepsilon_{\varphi\TT}
\leq
\left(
    \prod_{j=1}^d
    \norm{\mathbf{Q}_j}_2
\right)
\delta_{\varphi\TT},
\qquad
\varepsilon_{\varphi\TT}
:=
\norm{
    \hat{\mathcal F}
    -
    \hat{\mathcal F}^{\varphi}_{\mathbf r}
}_F .
\end{equation}

The factor $\prod_j\norm{\mathbf{Q}_j}_2$ is computable once $\mathbf D$,
$\bK$, $\bm{\Omega}$, $\mathbf M$, and the quadrature rule are fixed. It also
admits an explicit bound.
By \eqref{eq:Psi-separable},
$\Psi_j(\cdot,k)$ is the Fourier transform of
$\phi^{(j)}_k$ extended by zero outside $D_j$,
so $\mathbf{Q}_j\mathbf{Q}_j^{*}$ is a weighted analogue of a Plancherel
identity for the cosine basis, with the quadrature weights entering squared.
The step below factors out one power of the weights and leaves an ordinary
quadrature sum.
Fix $\mathbf{c}\in\mathbb{C}^{K_j}$ with $\norm{\mathbf{c}}_2=1$ and put
$g_c:=\sum_{k=0}^{K_j-1}c_k\phi^{(j)}_k$, so that
$\norm{g_c}_{L^2(D_j)}=\norm{\mathbf{c}}_2=1$
by orthonormality of the cosine basis.
Taking the same combination of the
$\Psi_j(\cdot,k)$ and using linearity of the integral
in \eqref{eq:Psi-separable} gives
\[
    \Psi_c(\omega)
    :=
    \sum_{k=0}^{K_j-1}
    c_k
    \Psi_j(\omega,k)
    =
    \int_{D_j}
    e^{-i\omega x}
    g_c(x)\,dx ,
\]
the Fourier transform of $g_c$ extended by zero outside $D_j$. This extension
allows the use of Plancherel's theorem below.
Combining the rows of $\mathbf{Q}_j$ with the same
$\mathbf{c}$ gives
$\bigl(\mathbf{c}^{\top}\mathbf{Q}_j\bigr)_m
=(2\pi)^{-1}w^{(j)}_m\Psi_c(\omega^{(j)}_m)$,
so that, summing over $m$,
\[
\begin{aligned}
\norm{\mathbf{c}^{\top}\mathbf{Q}_j}_2^2
&=
\frac{1}{4\pi^2}
\sum_m
\bigl(w^{(j)}_m\bigr)^2
|\Psi_c(\omega^{(j)}_m)|^2
\leq
\frac{w^{(j)}_{\max}}{4\pi^2}
\sum_m
w^{(j)}_m
|\Psi_c(\omega^{(j)}_m)|^2 ,
\end{aligned}
\]
with $w^{(j)}_{\max}:=\max_m w^{(j)}_m$.
The remaining sum is the $M_j$-point
Gauss--Legendre approximation of
$\int_{-\Omega_j}^{\Omega_j}|\Psi_c|^2\,d\omega$,
and Plancherel bounds the exact integral by
$\int_{\R}|\Psi_c|^2\,d\omega
=2\pi\norm{g_c}_{L^2(D_j)}^2=2\pi$.
Let
\begin{equation}
    \label{eq:Rquad}
    \mathcal R_j
    :=
    \sup_{\norm{\mathbf{c}}_2=1}
    \left(
        \sum_{m=1}^{M_j}
        w^{(j)}_m
        |\Psi_c(\omega^{(j)}_m)|^2
        -
        \int_{-\Omega_j}^{\Omega_j}
        |\Psi_c(\omega)|^2\,d\omega
    \right)
\end{equation}
be the worst-case quadrature remainder over the
unit sphere.
Since
$\norm{\mathbf{Q}_j}_2
=\sup_{\norm{\mathbf{c}}_2=1}\norm{\mathbf{c}^{\top}\mathbf{Q}_j}_2$,
maximizing over $\mathbf{c}$ then gives
\begin{equation}
\label{eq:Q-plancherel-bound}
    \norm{\mathbf{Q}_j}_2^2
    \leq
    \frac{w^{(j)}_{\max}}{4\pi^2}
    \left(
        2\pi+\mathcal R_j
    \right).
\end{equation}
The Plancherel constant $2\pi$ is independent of
$K_j$: it comes from orthonormality of the basis,
not from the number of retained terms.
The remainder $\mathcal R_j$ is a supremum over
$\mathbf{c}\in\mathbb{C}^{K_j}$, so it cannot decrease as
$K_j$ grows. Nevertheless, the following support argument identifies a
resolution scale with no explicit dependence on $K_j$.
Since $g_c$ is supported in $D_j$, $|\Psi_c|^2$ is
the Fourier transform of the autocorrelation of
$g_c$, which is supported in $[-|D_j|,|D_j|]$,
with $|D_j|:=u_j-l_j$.
Hence, $|\Psi_c|^2$ is entire of exponential type
$|D_j|$ for every unit $\mathbf{c}$, whichever cosine terms $\mathbf{c}$
combines. This identifies $\Omega_j|D_j|$, rather than $K_j$ directly, as the
natural resolution scale for the quadrature remainder.
The phase-rate guide \eqref{eq:M-centered-guide} of
Section~\ref{subsec:chf-difficulty} reached the same scaling by a different
route.
For Gauss--Legendre on $[-\Omega_j,\Omega_j]$ we
have $w^{(j)}_{\max}\approx\pi\Omega_j/M_j$, so in
the resolved regime, where
$\mathcal R_j\leq2\pi\varrho_j$ with $\varrho_j$ small,
\begin{equation}
\label{eq:Q-practical-bound}
    \norm{\mathbf{Q}_j}_2
    \lesssim
    \left(
        \frac{
            (1+\varrho_j)\Omega_j
        }{
            2M_j
        }
    \right)^{1/2}.
\end{equation}
Here and below, $\lesssim$ and $\gtrsim$ suppress positive constants independent
of $M_j$, $\Omega_j$, $|D_j|$ and $K_j$.

Thus, by \eqref{eq:Q-practical-bound}, with the
resolution scaling
$M_j\gtrsim\Omega_j|D_j|$,
\[
    \norm{\mathbf{Q}_j}_2
    =
    O(|D_j|^{-1/2}),
\]
up to quadrature-dependent constants. The scaling is stated as a lower rather
than an upper bound: $M_j$ enters
\eqref{eq:Q-practical-bound} in the denominator, so only a matching lower
bound on $M_j$ yields the decay in $|D_j|$.
The transfer therefore introduces no explicit adverse
dependence on $K_j$ once the quadrature resolves the
retained kernel terms.

The Plancherel step is what makes the bound decay. Without assuming that the
quadrature resolves the kernel, one can instead use the crude bound
$|\Psi_c|\leq|D_j|^{1/2}$, which follows from
Cauchy--Schwarz and $\norm{g_c}_{L^2(D_j)}=1$.
With $\sum_m w^{(j)}_m=2\Omega_j$ and the same
estimate for $w^{(j)}_{\max}$, this gives the weaker
$\norm{\mathbf{Q}_j}_2\leq
\Omega_j\bigl(|D_j|/(2\pi M_j)\bigr)^{1/2}$,
which exceeds \eqref{eq:Q-practical-bound} by a factor
$O\bigl((\Omega_j|D_j|)^{1/2}\bigr)$ and, under the same resolution scaling,
does not decay in $|D_j|$ at all, giving only $O(\Omega_j^{1/2})$.

Thus, the exact transfer bound \eqref{eq:chf-TT-transfer} is computable, while
\eqref{eq:Q-practical-bound} shows that, once the quadrature is resolved, the
transformation introduces no explicit adverse dependence on the number of
retained cosine terms.

\section{Expectation calculations}
\label{sec:basket}

Both tensorized COS methods yield the same tensorized representation of the
Fourier--cosine expansion and differ only in how its cores are computed offline,
so everything below applies to either method.

We therefore drop the method label and write $\hat{C}_j$ for the coefficient
cores, $\hat{\Gamma}_j$ for the functional cores and
$\hat{f}_{\mathbf{r}\text{-TT}}$ for the density they synthesize:
\begin{equation}
\label{eq:generic-cores}
\bigl(
\hat{C}_j,\;
\hat{\Gamma}_j,\;
\hat{f}_{\mathbf{r}\text{-TT}}
\bigr)
=
\begin{cases}
\bigl(
\hat{A}_j,\;
\hat{\Gamma}^{C}_j,\;
\hat{f}_{\mathbf{r}\text{-CTT}}
\bigr),
& \text{COS-TT, from
\eqref{eq:A-TT}, \eqref{eq:functional-core}, \eqref{eq:COS-TT-approx}},
\\[0.4em]
\bigl(
\hat{F}^{\varphi}_j,\;
\hat{\Gamma}^{\varphi}_j,\;
\hat{f}_{\mathbf{r}\text{-}\varphi\text{TT}}
\bigr),
& \text{COS-TT-CHF, from
\eqref{eq:chf-cores-quad}, \eqref{eq:functional-core-chf-quad},
\eqref{eq:COS-TT-CHF-approx-quad}}.
\end{cases}
\end{equation}
We assume throughout that $f$ satisfies the conditions of
Section~\ref{subsec:COS}, so that it admits a Fourier--cosine TT.

\subsection{Expectations of separable functions}
\label{sec:expectation}
We consider the problem of computing
$\E[g(\mathbf{X})] = \int_{\R^d} g(\bx) f(\bx)\,d\bx$
when the target function $g$ is separable:
\begin{equation}
\label{eq:separable-g}
g(\bx) = \prod_{j=1}^{d} g_j(x_j).
\end{equation}

We approximate $\E[g(\mathbf{X})]=\langle g,f\rangle$ by
$\langle g,\hat{f}_{\mathbf{r}\text{-TT}}\rangle$. Applying
Theorem~\ref{thm:FTT-inner-product} to the density FTT and the separable target
\eqref{eq:separable-g} gives
\begin{equation}
\label{eq:expectation-formula}
\langle g, \hat{f}_{\mathbf{r}\text{-TT}}\rangle
= \sum_{\balpha}^{\mathbf r}\prod_{j=1}^{d}
\langle g_j, \hat{\Gamma}_j(\alpha_{j-1};\cdot;\alpha_j)\rangle_{L^2(D_j)} .
\end{equation}
Expanding the functional cores in the cosine basis leaves the univariate
integrals
\[
V^{(j)}_{k_j}
:=\int_{l_j}^{u_j} g_j(x_j)\,\phi^{(j)}_{k_j}(x_j)\,dx_j ,
\]
the Fourier--cosine coefficients of the univariate target $g_j$ in coordinate
$j$. They depend only on $g_j$ and the basis, not on the density or TT cores,
and can be precomputed analytically or by one-dimensional quadrature. The
resulting approximation is
\begin{equation}
\label{eq:expectation-final}
\boxed{
\E[g(\mathbf{X})] \approx
\hat{\E}_{\mathbf{r}\text{-TT}}[g(\mathbf{X})]
:=
\sum_{\balpha}^{\mathbf r}\prod_{j=1}^{d}\sum_{k_j=0}^{K_j-1}
\hat{C}_j(\alpha_{j-1},k_j,\alpha_j)\, V^{(j)}_{k_j}.
}
\end{equation}
\begin{remark}[Error of the contracted expectation]
\label{rem:expectation-error}
$\hat{\E}_{\mathbf{r}\text{-TT}}[g(\mathbf{X})]
= \langle g,\hat{f}_{\mathbf{r}\text{-TT}}\rangle_{L^2(\mathbf D)}$ is computed
from the approximate density on $\mathbf D$, and the contraction itself
introduces no further approximation beyond those of the chain that produced
the train, \eqref{eq:approx-chain} for COS-TT and \eqref{eq:chf-chain} for
COS-TT-CHF. Assume that $gf\in L^1(\R^d)$ and $g\in L^2(\mathbf D)$, and choose
$\mathbf D$ such that
\[
\int_{\R^d\setminus\mathbf D}|g(\bx)|f(\bx)\,d\bx
\leq \tau_{g,\mathbf D}.
\]
The attainable tolerance $\tau_{g,\mathbf D}$ depends on the physical-domain
truncation range and the tail decay of $|g|f$. By Cauchy--Schwarz,
\[
\bigl|\E[g(\mathbf{X})]-\hat{\E}_{\mathbf{r}\text{-TT}}[g(\mathbf{X})]\bigr|
\leq \tau_{g,\mathbf D}+\norm{g}_{L^2(\mathbf{D})}\,
\norm{f-\hat{f}_{\mathbf{r}\text{-TT}}}_{L^2(\mathbf{D})} .
\]
Thus, apart from physical-domain truncation, the density approximation error is
amplified by at most $\norm{g}_{L^2(\mathbf D)}$. This bound applies to either
tensorized COS method.
\end{remark}

\subsection{Expectations of nonseparable functions}
\label{subsec:lincomb}

Although the contraction of Section~\ref{sec:expectation} requires a separable
target, tensorized COS methods can also treat certain nonseparable functions
efficiently. An important class, underlying many problems in quantitative
finance, consists of expectations $\E[G(H(\mathbf X))]$, where the components of
$\mathbf X$ are dependent, $G:\R\to\R$ is univariate, and the scalar aggregate
$H:\R^d\to\R$ is a sum of univariate functions,
\begin{equation}
\label{eq:scalar-aggregate}
H(\bx) = \sum_{j=1}^{d} h_j(x_j),
\end{equation}
with any weights absorbed into the $h_j$. Neither $H$ nor $G\circ H$ is
generally separable, so the target cannot be contracted against the density
train coordinate by coordinate. The exponential of $H$, however, is separable,
\begin{equation}
\label{eq:aggregate-separable}
e^{i\omega H(\bx)}
= \prod_{j=1}^{d}\exp\!\bigl(i\omega\,h_j(x_j)\bigr),
\qquad \omega\in\R .
\end{equation}
This identity supplies the separable target needed to contract the density TT.
Writing $\hat\varphi_H$ for the resulting approximation to the ch.f.\
$\varphi_H$ of $H(\mathbf X)$, and $\hat f_H$ for its one-dimensional COS
density, the complete reduction is
\[
\begin{gathered}
\varphi(\bomega)
\xrightarrow{\text{tensorized COS}}
\hat f_{\mathbf r\text{-TT}}(\bx)
\xrightarrow[\eqref{eq:aggregate-separable}]{\text{TT contraction}}
\hat\varphi_H(\omega)
\xrightarrow{\text{one-dimensional COS}}
\hat f_H(y)
\longrightarrow
\hat\E[G(H(\mathbf X))].
\end{gathered}
\]
The first arrow is the unchanged offline construction of
Section~\ref{sec:method}. The TT contraction is the only multidimensional step
that depends on $H$, while the remaining steps are the ordinary one-dimensional
COS computations recalled in Section~\ref{subsec:COS}.

For fixed $\omega$, define the unnormalized cosine integrals of the factors in
\eqref{eq:aggregate-separable},
\begin{equation}
\label{eq:aggregate-inner}
I_j(k_j,\omega) := \int_{l_j}^{u_j}
\exp\!\bigl(i\omega\,h_j(x)\bigr)
\cos\!\left(\frac{k_j\pi}{u_j-l_j}(x-l_j)\right)dx .
\end{equation}
Because $\phi^{(j)}_{k_j}$ is the corresponding normalized cosine, the
coordinate-wise contraction is $\beta^{(j)}_{k_j}I_j(k_j,\omega)$. Substitution
into \eqref{eq:expectation-final} gives the aggregate ch.f.,
\begin{equation}
\label{eq:aggregate-chf}
\begin{aligned}
\varphi_H(\omega)
&= \int_{\R^d} e^{i\omega H(\bx)}\, f(\bx)\, d\bx
\;\approx\;
\int_{\mathbf{D}} e^{i\omega H(\bx)}\,
\hat{f}_{\mathbf{r}\text{-TT}}(\bx)\, d\bx
\\[0.2em]
&= \sum_{\balpha}^{\mathbf r}\prod_{j=1}^{d}\sum_{k_j=0}^{K_j-1}
\hat{C}_j(\alpha_{j-1},k_j,\alpha_j)\,
\beta^{(j)}_{k_j}\, I_j(k_j,\omega)
\;=:\;\hat\varphi_H(\omega).
\end{aligned}
\end{equation}
The approximation restricts the integral to $\mathbf D$ and replaces $f$ by the
Fourier--cosine TT of \eqref{eq:generic-cores}. Evaluating
\eqref{eq:aggregate-chf} on a scalar COS frequency grid recovers $\hat f_H$ on
a one-dimensional truncation interval $[a_H,b_H]$, chosen as in
Section~\ref{subsec:COS}, and gives
\[
\E[G(H)]\approx\int_{a_H}^{b_H}G(y)\hat f_H(y)\,dy.
\]
No tensor representation of $G$ is required. For fixed $h_j$ and $\mathbf D$,
the inner integrals \eqref{eq:aggregate-inner} contain no density or ch.f.\
evaluations and can be precomputed, while the model enters only through the
coefficient cores. Changing $G$ affects only the final one-dimensional integral,
so the same density train and aggregate density $\hat f_H$ serve many target
functions.

\subsection{Application: semi-analytical formula for basket option pricing}
For the European basket options of Section~\ref{sec:numerics}, financial models
are written for the log-prices $\mathbf X := \ln(\mathbf S)$ directly, while the
basket is a weighted sum of the prices themselves,
\begin{equation}
\label{eq:basket-H}
H(\mathbf X) = \sum_{j=1}^{d} \vartheta_j e^{X_j} = \sum_{j=1}^{d} \vartheta_j S_j ,
\qquad S_j := e^{X_j},
\end{equation}
which is \eqref{eq:scalar-aggregate} with $h_j(x)=\vartheta_j e^{x}$. We write
$\mathcal K$ for the strike, to distinguish it from the expansion-term counts
$\bK$. Define the univariate payoff $G^{\pm}$ and the corresponding basket payoff
by
\begin{equation}
\label{eq:basket-payoff}
G^{\pm}(y;\mathcal K)
:=\bigl(\pm(y-\mathcal K)\bigr)^+,
\qquad
\Lambda^{\pm}(\mathbf X(T);\mathcal K)
:=G^{\pm}\!\left(H(\mathbf X(T));\mathcal K\right),
\end{equation}
where the upper sign denotes a call and the lower a put. The time-zero prices are
$\Pi^{\pm}(0;\mathcal{K}) =
e^{-r_fT}\E^{\mathbb{Q}}[\Lambda^{\pm}(\mathbf{X}(T);\mathcal{K})]$, where $T$
is the maturity and $r_f$ the risk-free rate. Here $G^{\pm}$ plays the
role of $G$, so changing the strike changes only the final one-dimensional
integration.

For this aggregate the inner integrals \eqref{eq:aggregate-inner} read
\begin{equation}
\label{eq:inner-integral-def}
I_j(k_j, \omega) := \int_{l_j}^{u_j} \exp(i\omega\vartheta_j e^{x_j})\,
\cos\!\left(\frac{k_j\pi}{u_j-l_j}(x_j - l_j)\right) dx_j .
\end{equation}
Their integrand contains a cosine factor with frequency $k_j\pi/(u_j-l_j)$, so
a quadrature rule accurate for all $k_j<K_j$ needs a resolution that grows with
$K_j$. Evaluating
\eqref{eq:inner-integral-def} numerically at every frequency of the scalar COS
grid would then dominate the computational cost. The following result removes
that cost by evaluating \eqref{eq:inner-integral-def} in closed form, in terms
of the incomplete gamma integral over a segment,
\begin{equation}
\label{eq:gamma-segment}
\gamma(s;z_a,z_b) := \int_{z_a}^{z_b} t^{s-1}e^{-t}\,dt ,
\end{equation}
along the straight path from $z_a$ to $z_b$. We use this two-endpoint form
rather than the usual lower incomplete gamma function
$\gamma(s,z)=\int_0^z t^{s-1}e^{-t}\,dt$ because the exponents $s$ arising below
are purely imaginary. The integrand then satisfies
$|t^{s-1}|=|t|^{-1}e^{-\operatorname{Im}(s)\arg t}$, which is a constant
multiple of $1/|t|$ along any fixed ray. Thus, the integral from the origin does
not converge and $\gamma(s,z)$ is available only
by analytic continuation, which is moreover singular at $s=0$. By contrast,
\eqref{eq:gamma-segment} converges for every $s$ whenever the path avoids the
origin and the branch cut, and it agrees with $\gamma(s,z_b)-\gamma(s,z_a)$
whenever the latter is defined.

\begin{theorem}[Closed form for the inner integrals]
\label{thm:inner-integrals}
Let $\eta_j = k_j\pi/(u_j-l_j)$, let $z^s := \exp(s\log z)$ denote the principal
branch, and let $\operatorname{Ei}(x) = -\int_{-x}^{\infty} e^{-t}/t\,dt$ be the
exponential integral, taken on its principal branch for the complex arguments
below. Only differences of its values are used, and such a difference is the
convergent integral $\operatorname{Ei}(z_b)-\operatorname{Ei}(z_a)
= \int_{z_a}^{z_b} e^{t}/t\,dt$ along the segment joining the two points.
Then
\eqref{eq:inner-integral-def} evaluates as follows.

\emph{(i) Degenerate phase.} If $\omega\vartheta_j=0$, then
\[
I_j(k_j,\omega)
=
\begin{cases}
0, & k_j\neq0,\\
u_j-l_j, & k_j=0.
\end{cases}
\]

\emph{(ii) Zeroth cosine term.}
If $\omega\vartheta_j\neq0$ and $k_j=0$, then
\[
I_j(0,\omega)
=
\operatorname{Ei}(i\omega\vartheta_j e^{u_j})
-
\operatorname{Ei}(i\omega\vartheta_j e^{l_j}).
\]

\emph{(iii) General case.}
If $\omega\vartheta_j\neq0$ and $k_j\neq0$, then
\begin{equation}
\label{eq:Ipm}
\boxed{
I_j(k_j,\omega)
=
\tfrac{1}{2}I_j^{(+)}(k_j,\omega)
+
\tfrac{1}{2}I_j^{(-)}(k_j,\omega),
\qquad
I_j^{(\pm)}
=
e^{\mp i\eta_j l_j}
(-i\omega\vartheta_j)^{\mp i\eta_j}
\gamma\bigl(
    \pm i\eta_j;
    \zeta_j^{a},
    \zeta_j^{b}
\bigr),
}
\end{equation}
where
\[
\zeta_j^{a}=-i\omega\vartheta_j e^{l_j},
\qquad
\zeta_j^{b}=-i\omega\vartheta_j e^{u_j}.
\]
Both endpoints lie on the ray of direction
$-i\operatorname{sgn}(\omega\vartheta_j)$, so the segment joining them
meets neither the origin nor the principal cut $(-\infty,0]$.
\end{theorem}

\begin{proof}
Omit the subscript $j$ and write
$I=\int_a^b e^{iqe^x}\cos\bigl(\eta(x-a)\bigr)\,dx$, with
$q:=\omega\vartheta$ and $\eta:=k\pi/(b-a)$.

\emph{(i)} If $q=0$, then $I=\int_a^b\cos\bigl(\eta(x-a)\bigr)\,dx$, which
equals $\eta^{-1}\sin(k\pi)=0$ for $k\neq0$ and $b-a$ for $k=0$.

\emph{(ii)} If $q\neq0$ and $\eta=0$, the cosine is unity, and the substitution
$u=e^x$ gives
$I=\int_{e^a}^{e^b}e^{iqu}/u\,du=\operatorname{Ei}(iqe^x)\big|_a^b$.

\emph{(iii)} Let $q\neq0$ and $\eta\neq0$. Writing
$\cos\bigl(\eta(x-a)\bigr)=\tfrac12\sum_{\pm}e^{\pm i\eta(x-a)}$ gives
$I=\tfrac12(I_++I_-)$, where the substitution $u=e^x$ yields
$I_\pm=e^{\mp i\eta a}\int_{e^a}^{e^b}u^{\pm i\eta-1}e^{iqu}\,du$.
For $s=\pm i\eta$, the further substitution $t=-iqu$, which maps
$[e^a,e^b]$ onto the segment from $\zeta^a$ to $\zeta^b$, gives
$\int_{e^a}^{e^b}u^{s-1}e^{iqu}\,du=(-iq)^{-s}\gamma\bigl(s;\zeta^a,\zeta^b\bigr)$,
the powers taken on the principal branch. Both sides are finite because the
left-hand integrand is continuous on $[e^a,e^b]$ and the right-hand path
avoids the origin. Substituting $s=\pm i\eta$ into this identity yields
\eqref{eq:Ipm}.
\end{proof}

\begin{remark}[Numerical evaluation]
\label{rem:gamma-eval}
The quantities in \eqref{eq:Ipm} are complex-valued and enter the
construction of the aggregate ch.f. The exponent $\pm i\eta_j$ is purely
imaginary, so evaluation requires a library that supports complex parameters.
Use a routine that takes both endpoints and so returns \eqref{eq:gamma-segment}
directly, such as \texttt{mpmath.gammainc(s, a, b)}; forming the difference of
two separately computed values of the continued $\gamma(s,z)$ is neither
necessary nor numerically preferable. The principal branch must be used
consistently in $(-i\omega\vartheta_j)^{\mp i\eta_j}$ and $\gamma$. As in
\eqref{eq:Ipm}, it is
preferable to evaluate $I_j^{(+)}$ and $I_j^{(-)}$ separately and then sum them.
\end{remark}

\section{Numerical results}
\label{sec:numerics}
\label{subsec:basket-numerics}

Each sensitivity study below varies the stated controls while
holding the remainder fixed, and reports two quantities: the joint density,
which is what the tensorized COS methods construct, and a European basket call,
which exercises the expectation calculation of Section~\ref{sec:basket}. The two
weight the physical-domain truncation range differently, as
Section~\ref{subsec:num-K-rank} shows.



\subsection{Setup}
\label{subsec:num-setup}

\paragraph{Models.}
We take basket option pricing as the application example and assess both methods
under GBM and VG, which exhibit rapid and algebraic ch.f.\ decay, respectively.
We work throughout in log-asset coordinates,
$\mathbf{X}=(\log S_1,\dots,\log S_d)$. Under GBM, the ch.f.\ of $\mathbf{X}$ is
Gaussian. Both models are multidimensional L\'evy processes, so each ch.f.\
follows from the L\'evy--Khintchine representation. See
\citet{cont2004financial} for the general form and
\citet{luciano2006multivariate} for the multivariate VG construction based on a
shared gamma subordinator.

\paragraph{Model parameters.} We follow Table~4 of
\citet{arenstein2026costtchf}, allowing direct comparison with that independent
replication of our methods. In both models, $S_0 = 100$, the basket weights are
equal, $T = 1$, and volatility is
heterogeneous across dimensions, $\sigma_j = \sigma_{\min} + (\sigma_{\max} -
\sigma_{\min})x_j + \varepsilon\sin(\pi x_j)$ with $x_j = (j-1)/(d-1)$. For GBM
$(\sigma_{\min},\sigma_{\max},\varepsilon) = (0.18, 0.30, 0.015)$ with
$\rho_{ij} = 0.7^{|i-j|}$ and $r_f = 0.02$. For VG, the profile is
$(0.20, 0.40, 0.05)$ with
$\rho_{ij} = 0.35^{|i-j|}$, $r_f = 0.03$, $\theta_j = -0.30$ and $\nu = 0.10$.
The basket call has strike $\mathcal{K} = 100$. Heterogeneous volatilities avoid
the permutation invariance induced by identical marginals and equal weights,
which would yield an unrepresentative coefficient-tensor rank.

All experiments set the three controls we choose isotropically, $K_j = K$,
$\Omega_j = \Omega$ and $M_j = M$ for every $j$, and we write the scalars $K$,
$\Omega$ and $M$ for these common values below, reserving $\bK$,
$\bm{\Omega}$ and $\mathbf{M}$ for the general anisotropic case. The
physical-domain truncation range is not among them: Remark~\ref{rem:domain-rule}
applies the cumulant rule coordinatewise, so $|D_j|$ inherits the heterogeneity
of $\sigma_j$ and is a vector even when everything we choose is a scalar. At
$d = 4$ the VG ranges span a factor $1.82$ between the narrowest and the widest
coordinate, and the GBM ranges at $d = 10$ a factor $1.67$. This matters
wherever a scalar control has to satisfy a per-coordinate condition, since it is
then the worst coordinate that binds, and the worst coordinate is not the same
one for the two frequency conditions; Section~\ref{subsec:num-window} makes that
explicit.

The frequency-domain truncation-range study in
Section~\ref{subsec:num-window} uses, for its GBM half, a homogeneous set of
parameters, $\sigma_j = 0.4$, $\rho_{ij} = 0.5$ and $r_f = 0$, for which
$|D_j| = 2\lambda_1\sigma\sqrt{T}$ is common to every coordinate and the
compatibility bound $\kappa_j$ of \eqref{eq:Omega-compatibility} is available in
closed form. Because that study concerns the truncation range rather than the
rank, homogeneity does not bias its conclusions. Its VG half uses the
heterogeneous parameters above.

\paragraph{Reference values.}
We evaluate the density at $10^3$ points drawn from the model and the basket
call at $\mathcal{K}=100$.\footnote{Based on \citet{fang2009novel}, the
convergence rate does not depend on the choice of strike price, which affects
the error level only slightly. We therefore consider this simplification
sufficient for our testing purposes.} In the GBM case, we work in log-asset
coordinates, so
the joint density is Gaussian and the reference density is the exact
multivariate Normal in every dimension.

The VG density and both basket prices lack closed forms, and Monte Carlo cannot
attain the required accuracy. The measured errors reach $10^{-13}$, whereas a
randomized-Sobol estimator with a geometric-basket control variate has a
half-width between $10^{-5}$ and $5\times10^{-4}$ over the reported runs. To
study convergence, we therefore use the method's output at $K = 512$ and
$\lambda_1 = 16$ as the reference value, well above the experimental maxima of
$192$ and $10$. Four neighbouring refined settings agree with this value to
within $10^{-14}$ in mean relative density. We therefore treat it as exact to a
resolution of $10^{-13}$ and mark that level on panels approached by the VG
curves.

The Sobol estimator independently validates the price. Across the $401$ runs of
Sections~\ref{subsec:num-K-rank}--\ref{subsec:num-cost}, spanning $15$
independent market--dimension comparisons, the reference value agrees with the
estimator to within $1.19$ half-widths. We plot its uncertainty
band but do not use it as the reference because its sampling error would impose
an artificial floor on every curve. The sole exception is
Table~\ref{tab:ladder-accuracy}, whose ladder reaches dimensions at which a
converged reference of our own is unaffordable, so its price errors are
measured against the estimator instead.
Neither reference check would detect an
error common to all four model--method combinations, particularly a rank error.
We therefore verify rank convergence separately in
Section~\ref{subsec:num-K-rank}.

\paragraph{Error measures and plotting conventions.}
The following conventions apply unless a figure caption states otherwise.
The density error is measured as the mean relative error over the $10^3$ points
above and the price error as the relative error of the basket call at
$\mathcal{K}=100$, both computed from the same coefficient trains. Where the two
methods are compared, filled markers denote COS-TT and open markers COS-TT-CHF.
Where shown, shading marks the region in which a point resolves the reference
rather than the method. On density panels, it denotes the $10^{-13}$ resolution
of the VG reference value. On price panels, it denotes the quasi-Monte Carlo
half-width, which measures the estimator's uncertainty rather than the error of
the reference value. Figure~\ref{fig:omega} instead shades
$\Omega<\kappa$, as stated in its caption.
All tests reported in this paper were conducted on a laptop with an AMD Ryzen 7
PRO 8840HS processor (8 cores). All reported timings are single-threaded.

Each curve decreases before reaching a plateau. This plateau indicates that the
varied error component has fallen below another source, not that the error is
intrinsically insensitive to the control. Further refinement has no effect until
the dominant source is reduced.

\subsection{Frequency-domain truncation and discretization}
\label{subsec:num-window}

Proposition~\ref{prop:Omega-selection} makes the frequency-domain truncation
range a function of the cosine expansion rather than a free parameter. This
section tests that prediction, and then the resolution guide that follows from
it.

By a \emph{configuration} we mean a pair $(K,\lambda_1)$. It fixes the
expansion-term count $K$ and, through the rule of
Remark~\ref{rem:domain-rule}, the physical-domain widths $|D_j|$. It therefore
also fixes the compatibility bound of \eqref{eq:Omega-compatibility}. Within each
configuration we vary the truncation range along the ladder
\begin{equation}
\label{eq:Omega-ladder}
  \Omega/\kappa \in \{0.5,\,0.8,\,1,\,1.2,\,1.5\},
  \qquad
  \kappa := \max_j \kappa_j = \frac{(K-1)\pi}{\min_j |D_j|},
\end{equation}
and, independently of $\Omega$, the quadrature point count, at four values per
window. Both controls are scalars while $\mathbf{D}$ is not, and each binds on a
different coordinate: $\Omega$ must clear $\max_j \kappa_j$, which by
\eqref{eq:kappa-definition} belongs to the narrowest box, and $M$ must reach
$\max_j \chi_j = \Omega \max_j |D_j| / 2$ of \eqref{eq:M-centered-guide}, which
belongs to the widest. Both are lower bounds, so sizing each control for its own
worst coordinate over-resolves the others at a cost in time rather than
accuracy, which is what keeps them scalar. For the homogeneous GBM market the
two coincide and $M \in \{64,128,256,512\}$. The VG boxes differ by a factor
$1.82$, so a fixed set of counts would meet its six configurations at six
different resolutions; we therefore ladder the VG count in the normalized
resolution $M/\chi$ of \eqref{eq:M-over-chi}, requiring
$M/\chi \geq \{0.32,\,0.64,\,1.27,\,2.55\}$ with $M$ rounded up to a power of
two, so that every window is measured at two counts at or above $M_j = \chi_j$.
Each configuration thus contributes twenty runs.
Using the same absolute truncation ranges for all
configurations would produce arbitrary values of $\Omega/\kappa$ and obscure the
alignment predicted by \eqref{eq:Omega-compatibility}. We use nine GBM
configurations at $d = 15$, with $K \in \{32,48,64\}$ and
$\lambda_1 \in \{6,8,10\}$, and six VG configurations at $d = 4$, with
$K \in \{64,128\}$ and the same three $\lambda_1$. The VG term counts are larger
because VG requires more terms for comparable accuracy.

\begin{figure}[htbp]
  \centering
  \includegraphics[width=0.90\textwidth]{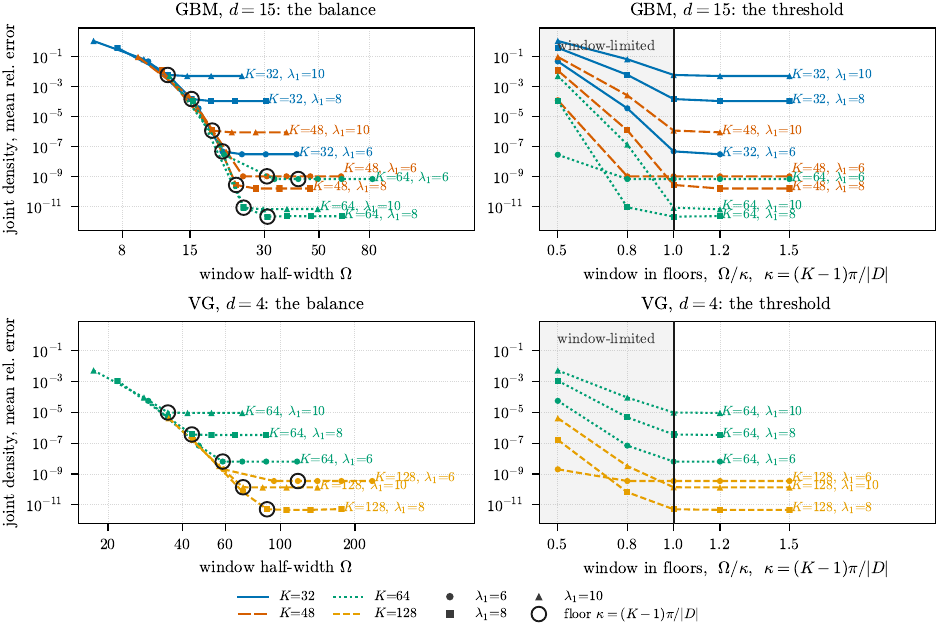}
  \caption{Relative error of the joint density against the frequency-domain
    truncation range.
    \emph{Left:} against the half-width $\Omega$ itself, minimized over
    quadrature point counts to exclude quadrature error. Open rings mark the
    configuration-specific compatibility bounds $\kappa$. \emph{Right:} the same
    runs against $\Omega/\kappa$, with
    shading for $\Omega < \kappa$. Top row GBM at $d = 15$, bottom VG at $d = 4$.
    Colour and dash pattern encode $K$, markers encode $\lambda_1$, and both are
    repeated at each curve's end.}
  \label{fig:omega}
\end{figure}

Each of the 15 curves in Figure~\ref{fig:omega} follows one configuration along
the ladder \eqref{eq:Omega-ladder}. In the left panel, plotted against $\Omega$
itself, the curves reach their plateaus at 14 different values of $\Omega$
spanning $12.2$ to $88.2$. In the right panel, after each curve is rescaled by
its own $\kappa$, the transitions align on the bound: every one of the 15
configurations is within a factor $1.8$ of its own plateau at $\Omega = \kappa$
and within $1.1$ at $\Omega = 1.2\,\kappa$, whereas at $\Omega = 0.8\,\kappa$
they stand up to $2.0\times10^{4}$ above it and at $\Omega = 0.5\,\kappa$ up to
$7.3\times10^{8}$. So $\kappa$ sets where a curve plateaus. The level at which it
plateaus is set instead by $K$ and $\lambda_1$, and spans many orders of
magnitude in both models.

Three of the 15 are already at their plateau at $\Omega = 0.8\,\kappa$, all
with $\lambda_1 = 6$, the narrowest range in the sweep. There the
physical-domain truncation error, which $\Omega$ does not enter, already
exceeds what the window discards, so the bound is not what those curves run
into.

This behaviour confirms \eqref{eq:Omega-compatibility}: the bound locates the
truncation range beyond which further enlargement no longer reduces the error,
replacing configuration-specific tuning with a single rule. Because $\kappa$
grows with $K$, increasing the expansion-term count at fixed $\Omega$ eventually
violates the compatibility condition, after which additional terms provide no
benefit. The truncation range and expansion-term count must therefore increase
together.

We now turn to the second control, the quadrature point count. Below the
compatibility bound, $\varepsilon_{\bm{\Omega}}$ dominates, because additional
quadrature points cannot recover information discarded by truncation. The error
is therefore insensitive to the point count there. Above the bound the point
count does matter, and \eqref{eq:M-centered-guide} states what it must be
measured against. The truncation range $\mathbf{D}$ of Remark~\ref{rem:domain-rule}
is centred on the marginal mean, so the dimensionless frequency
\eqref{eq:chi-def} of the kernel is $\chi_j=\Omega_j|D_j|/2$ and the guide reads
$M_j\gtrsim\chi_j$. The natural variable is therefore not $M_j$ but the
normalized resolution
\begin{equation}
\label{eq:M-over-chi}
  \frac{M_j}{\chi_j} \;=\; \frac{2M_j}{\Omega_j|D_j|},
\end{equation}
at whose unit value Gauss--Legendre places two points per oscillation at the
window centre, the position of poorest resolution in
Figure~\ref{fig:chf-oscillations}.

Two filters isolate $\varepsilon_{\mathrm{quad}}$ from the other error sources.
We drop the runs with $\Omega<\kappa$, where $\varepsilon_{\bm{\Omega}}$
dominates, and the runs whose error is within ten times the configuration's own
plateau, where $\varepsilon_{\bK}$ does. The second filter removes the three GBM
configurations with $K = 32$ outright. For GBM, the Gaussian
residual ch.f.\ is nonoscillatory, so the guide applies directly. For skewed VG,
the residual ch.f.\ contributes additional phase variation, and the same ratio
provides only a baseline normalization.

Figure~\ref{fig:nodes} confirms that $M_j/\chi_j$, rather than $M_j$ alone,
governs the transition. Call a curve flat once it comes within a factor of two
of its own error floor. Across the six GBM configurations, flatness sets in
between $M_j/\chi_j=0.86$ and $1.16$. The whole range lies within $16\%$ of the
Gauss--Legendre Nyquist point $M_j=\chi_j$, so the guide is calibrated for GBM.

\begin{figure}[!htbp]
  \centering
  \includegraphics[width=0.90\textwidth]{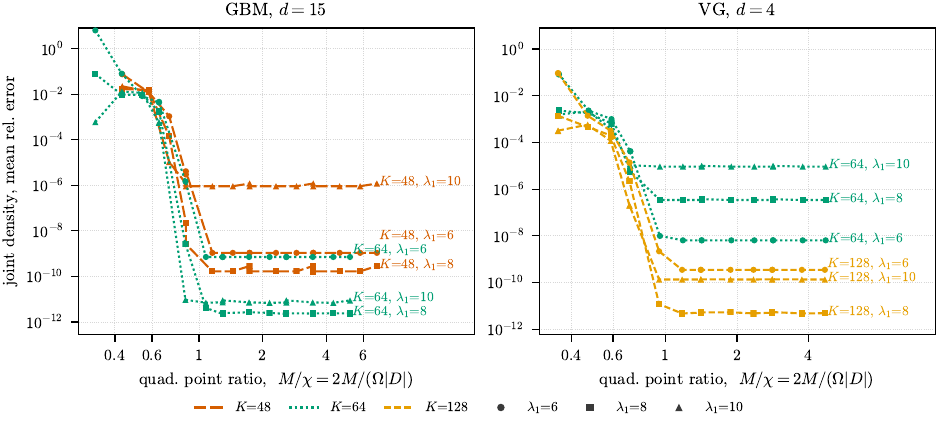}
  \caption{Relative error of the joint density against the normalized
    quadrature resolution \eqref{eq:M-over-chi}, for the same two studies and
    with the same conventions as
    Figure~\ref{fig:omega}. Runs with $\Omega < \kappa$ are
    excluded, as are the three GBM configurations with $K = 32$, whose cosine
    truncation error exceeds the quadrature error at every count. Six
    configurations remain per panel.}
  \label{fig:nodes}
\end{figure}

The six VG configurations flatten over $M_j/\chi_j\in[0.71,1.17]$, a range
that brackets $M_j=\chi_j$ as the GBM range does, within $30\%$ rather than
$16\%$. Twenty-four values of the ratio are reached by more than one
$(M,\Omega)$ pair, because the VG point counts are laddered in resolution
rather than in $M$, so that different windows meet at the same ratio; at each of
them the pairs agree to within $10\%$, and to $2\%$ at the median. This supports
the normalization.

One caveat remains on the low end of the VG range. The VG floors in
Figure~\ref{fig:omega} lie above the GBM floors because the nonsmooth VG density
yields algebraic rather than exponential Fourier--cosine convergence and
requires many more terms for a given accuracy \citep{fang2009novel}. The lowest
of the six, $M_j/\chi_j=0.71$ at $K=64$, $\lambda_1=10$, belongs to the
configuration whose floor is $9.3\times10^{-6}$: there $\varepsilon_{\bK}$
dominates and the curve plateaus before the quadrature binds, so that value is a
lower bound on the required resolution rather than a measurement of it.

The predicted change of regime is visible in both models. Read from right to
left, a few curves in Figure~\ref{fig:nodes} stop decreasing and rise again as
the point count is reduced. Every such rise starts below $M_j=0.6\,\chi_j$, near
the theoretical threshold $(e/4)\chi_j\approx0.68\,\chi_j$ of
Section~\ref{subsec:chf-difficulty}. Below it the grid no longer resolves the
kernel, and aliasing replaces ordinary quadrature error. This pre-asymptotic
regime does not support a convergence-rate estimate. Allowing a margin above the
largest transition observed in either model, $M_j/\chi_j=1.17$, gives the
conservative recommendation $M_j\geq1.3\,\chi_j$.

\subsection{The physical-domain truncation range, the expansion-term count and the rank}
\label{subsec:num-K-rank}

Figure~\ref{fig:k-l1} shows the fast convergence of the Fourier--cosine
expansion and the effect of varying $(K,\lambda_1)$, for both
methods and models. The picture is the classical COS trade-off
\citep{fang2009novel}: $K$ and $\lambda_1$ have to be balanced. A wide
truncation range pushes the relevant frequencies higher, so more terms are
needed to resolve them; a narrow range cuts off too much density mass, and that
truncation error dominates instead. The agreement between COS-TT-CHF and COS-TT
shows that, once the frequency-domain parameters satisfy our guide, COS-TT-CHF
attains the same high accuracy as COS-TT. Together,
Figures~\ref{fig:chf-oscillations}--\ref{fig:k-l1} show that the guide reduces
the remaining error control to the familiar COS choices of $\bK$ and
$\lambda_1$.

\begin{figure}[!htbp]
  \centering
  \includegraphics[width=0.90\textwidth]{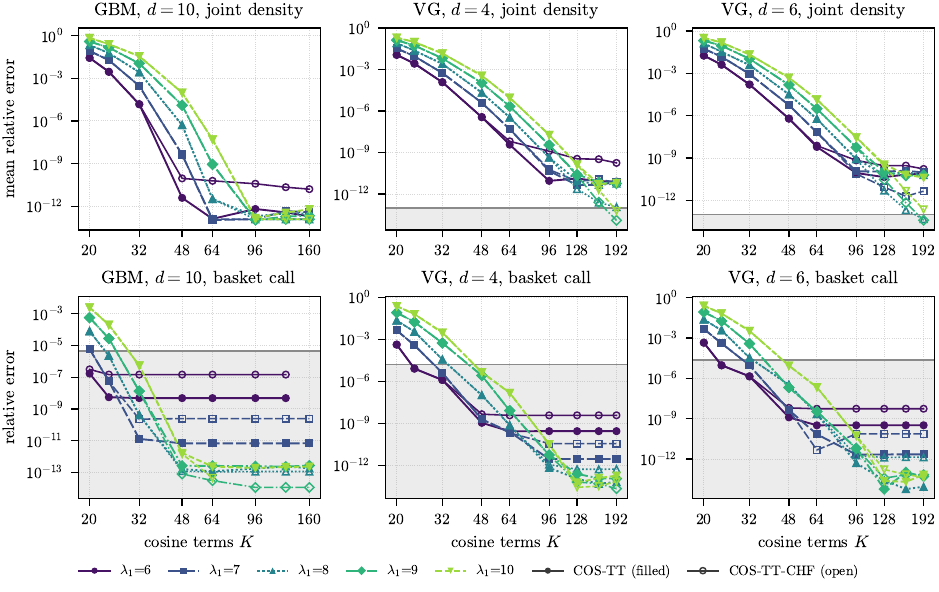}
  \caption{Sensitivity to the expansion-term count $K$ and the width multiplier
    $\lambda_1$ of Remark~\ref{rem:domain-rule}, for GBM at $d = 10$ and VG at
    $d = 4$ and $6$, at $\Omega = 80$ and $\Omega = 100$, respectively. The top
    row shows the relative error of the joint density, the bottom row that of
    the European basket call price at $\mathcal{K} = 100$. The
    horizontal axis is $K$. Colour and dash pattern encode
    $\lambda_1 \in \{6,\ldots,10\}$.}
  \label{fig:k-l1}
\end{figure}

\begin{figure}[htbp]
  \centering
  \includegraphics[width=0.90\textwidth]{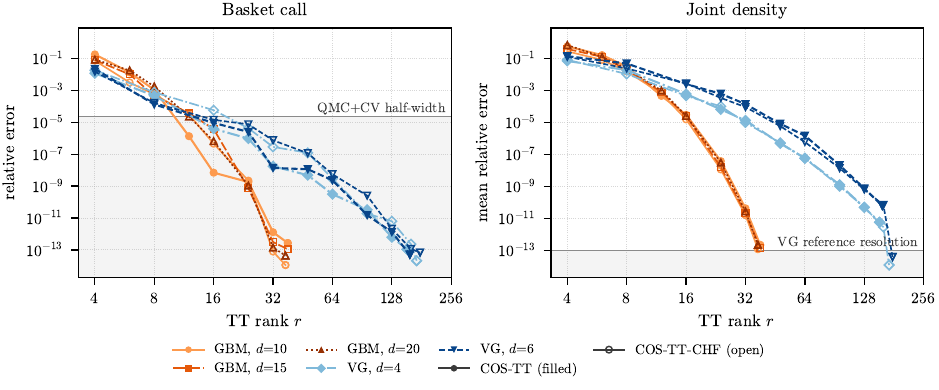}
  \caption{Relative errors against realized TT rank, for the basket call (left)
    and the joint density (right), for both methods at one configuration per
    model: $(K,\lambda_1) = (160,9)$ at $\Omega = 80$ for GBM, $(192,9)$ at
    $\Omega = 100$ for VG. The rank is selected by DMRG-greedy rather than
    prescribed, with its cap raised from $4$ to $64$ for GBM and $4$ to
    $224$ for VG. Table~\ref{tab:rank-converged} records where each series
    stops.}
  \label{fig:rank}
\end{figure}

Theorem~\ref{thm:train-conversion} and Assumption~\ref{ass:training} do not
determine the attainable rank, which must therefore be assessed numerically.
Figure~\ref{fig:rank} compares the methods at one common configuration per model,
using the narrowest range for which their density errors agree within the shared
truncation floor. Below $\lambda_1=9$, the error ratio is irregular because of
the width of $\mathbf{D}$ rather than the rank. The ch.f.-side formula omits the
truncated tail, whereas the coefficient-side formula folds it into
$\mathbf{D}$.

The required rank depends more strongly on the model than on dimension. GBM
converges at ranks $37$--$38$ for $d = 10$, $15$, and $20$, whereas VG already
requires ranks $153$--$178$ at $d = 4$ and $6$. The high VG ranks are consistent
with the shared gamma subordinator, which does not factorize across dimensions. The
two methods converge at comparable ranks: COS-TT-CHF requires approximately
$12\%$ higher ranks for VG and marginally lower ranks for GBM.
The analysis allows a rank penalty for COS-TT-CHF because the core-wise map
\eqref{eq:chf-cores-quad} is neither square nor orthogonal, so the best
rank-$\mathbf{r}$ approximation of $\varphi$ on the frequency grid
need not map to the best rank-$\mathbf{r}$ approximation of the coefficient
tensor. These measurements neither confirm nor exclude such a penalty, which
would be detectable only at high ranks. No curve in Figure~\ref{fig:rank} is
limited by the rank cap. Each approaches its configuration's truncation-error
floor.

\begin{table}[htbp]
  \centering
  \caption{Rank at which DMRG-greedy stops on its stopping tolerance of
    $10^{-13}$, with the relative density and price errors attained there, at
    the common configuration of Figure~\ref{fig:rank}. COS-TT was not run at
    GBM $d = 15$ and $20$.}
  \label{tab:rank-converged}
  \small
  \begin{tabular}{@{}l rll rll@{}}
    \toprule
    & \multicolumn{3}{c}{COS-TT-CHF} & \multicolumn{3}{c}{COS-TT} \\
    \cmidrule(lr){2-4}\cmidrule(lr){5-7}
    Case & Rank & Density & Price & Rank & Density & Price \\
    \midrule
    GBM $d=10$ & 37  & $1.2\times10^{-13}$ & $1.1\times10^{-14}$
               & 38  & $2.2\times10^{-13}$ & $2.7\times10^{-13}$ \\
    GBM $d=15$ & 38  & $1.4\times10^{-13}$ & $1.1\times10^{-13}$
               &     &                     &                     \\
    GBM $d=20$ & 37  & $2.2\times10^{-13}$ & $4.4\times10^{-14}$
               &     &                     &                     \\
    VG $d=4$   & 171 & $1.2\times10^{-14}$ & $2.1\times10^{-14}$
               & 153 & $5.7\times10^{-12}$ & $1.1\times10^{-13}$ \\
    VG $d=6$   & 178 & $3.9\times10^{-14}$ & $6.8\times10^{-14}$
               & 158 & $6.2\times10^{-11}$ & $5.0\times10^{-14}$ \\
    \bottomrule
  \end{tabular}
\end{table}


\subsection{Cost and dimensional scaling}
\label{subsec:num-cost}

\begin{figure}[htbp]
  \centering
  \includegraphics[width=0.94\textwidth]{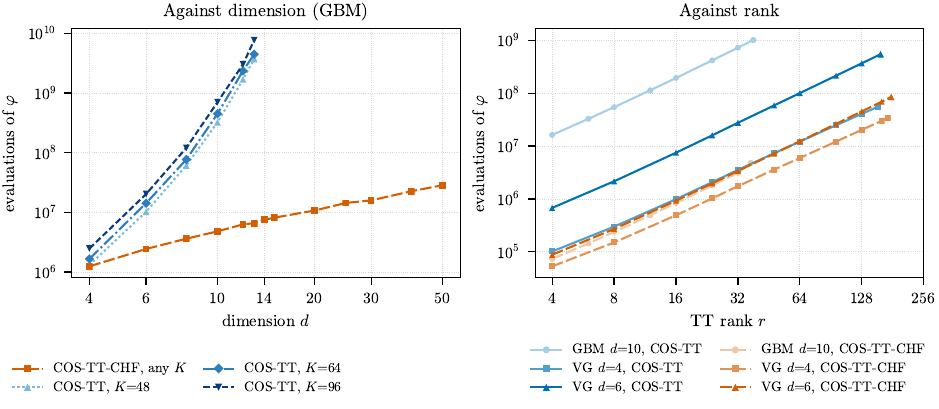}
  \caption{Number of evaluations of $\varphi$ requested by the
    compression. \emph{Left:} dimension scaling for the GBM runs of
    Figure~\ref{fig:dimension}, with $\Omega = 80$, $\lambda_1 = 7$, and a rank
    cap of $64$. The three
    COS-TT lines are $K = 48$, $64$ and $96$, and the single COS-TT-CHF line is
    common to all three.
    \emph{Right:} scaling with realized TT rank for both models and methods in
    Figure~\ref{fig:rank}.}
  \label{fig:samples}
\end{figure}

Because cross interpolation accesses a tensor only through entry evaluations,
the requested number of evaluations of $\varphi$ provides an
implementation-independent measure of offline cost. Each COS-TT-CHF entry
requires one such evaluation. By contrast, each COS-TT entry of $\calA$ sums
over the $2^{d-1}$ sign vectors of the mirrored domain through
\eqref{eq:multi-COS-coeff}. Figure~\ref{fig:samples} shows the resulting counts.
At fixed controls, the COS-TT-CHF count grows nearly linearly with $d$, by a
factor of $23$ from $d = 4$ to $d = 50$. The COS-TT count grows by more than
three orders of magnitude from $d = 4$ to $d = 13$. At fixed dimension, both
counts grow at the same rate with realized rank but differ by a
dimension-dependent factor. These observations agree with the analysis: both
methods request the same $\mathcal{O}(dnr^2)$ number of entries, as in
Section~\ref{subsec:TT-algorithms}, and only the cost per entry differs.

Figure~\ref{fig:time} separates wall time into offline compression and online
evaluation. Both methods use the same DMRG-greedy implementation under identical
conditions and differ only in the tensor being compressed. The offline panel
reproduces the sample-count trends. COS-TT-CHF grows algebraically, with fitted
exponents of $2.34$ for GBM and $2.24$ for VG, both below third order in $d$.
COS-TT instead grows by factors of $2.4$ per dimension for GBM and $2.5$ for
VG. This exponential growth comes from \eqref{eq:multi-COS-coeff}, not the
decomposition algorithm. The exponents above one for COS-TT-CHF do not
contradict the linear tensor complexity of
Section~\ref{subsec:TT-algorithms}, which counts requested entries. Each
COS-TT-CHF entry requires a ch.f.\ evaluation involving the quadratic form
$\bomega^\top\boldsymbol{\Sigma}\bomega$ and therefore costs
$\mathcal{O}(d^2)$. Consequently, the gap between algebraic and exponential
growth widens with dimension. 
COS-TT-CHF continues to $d = 150$ for GBM and
$d = 100$ for VG, requiring $745$\,s and $1361$\,s, respectively.\footnote{These
are the dimensions at which we chose to stop, not limits of the method: the
offline cost grows algebraically rather than exponentially in dimension, and the
errors in Table~\ref{tab:ladder-accuracy} remain small at the largest dimensions
tested. We fixed the range by practical relevance rather than by feasibility,
working to a single-threaded offline budget of about $30$ minutes per run. A
production risk system can absorb an offline stage of that size, whereas daily
recalculation would require a faster one.}

\begin{figure}[htbp]
  \centering
  \includegraphics[width=0.90\textwidth]{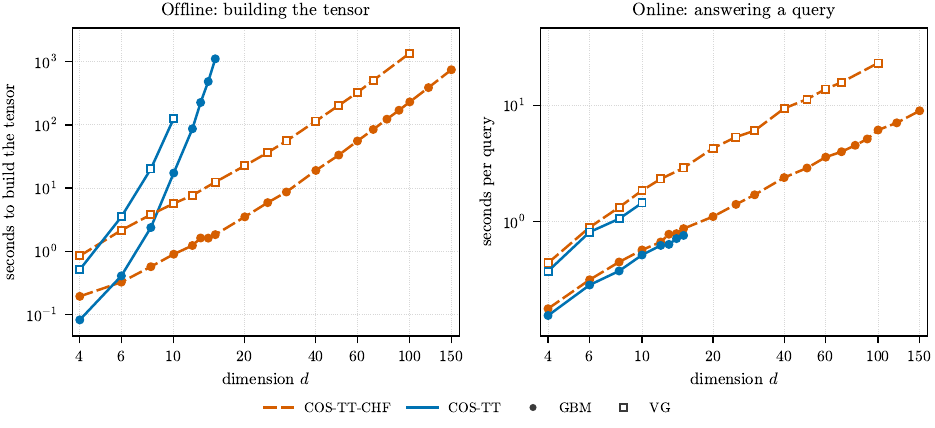}
  \caption{Wall time against dimension, split into the offline compression and
    the online evaluation. An online evaluation comprises the core-wise map, the
    basket call at five strikes and the joint density at $10^3$ points. Timings
    are single-threaded with the cache disabled and show the minimum of three
    repeats. The controls target a relative joint-density error of order
    $10^{-4}$ rather than $10^{-13}$. GBM uses $K = 64$,
    $\lambda_1 = 9$, $\Omega = 44$, $M = 195$, and rank cap $24$. VG uses $K = 96$,
    $\lambda_1 = 9$, $\Omega = 39$, $M = 296$, rank cap $48$. Colour and dash
    identify the method, and markers identify the model. Each curve is annotated
    with its fitted scaling. Both panels share a vertical scale.}
  \label{fig:time}
\end{figure}

Table~\ref{tab:ladder-accuracy} reports the corresponding relative errors. Two
references are used. The density error is reported for GBM only, against the
exact Normal density of the log-asset prices; the VG density has no closed
form at these dimensions. The price error is reported for both models,
against a randomized-Sobol quasi-Monte Carlo reference.

The GBM density error stays at or below the targeted order of $10^{-4}$ at
every dimension. The price errors stay below $1.5$ times the Sobol
half-width.

\begin{table}[!htbp]
  \centering
  \caption{Offline times and relative errors at selected dimensions of
    Figure~\ref{fig:time}. The density error is measured against the exact
    Normal density, available for GBM only. The price errors of both models are
    measured against a randomized-Sobol reference, whose relative half-width is
    listed alongside.}
  \label{tab:ladder-accuracy}
  \small
  \begin{tabular}{@{}r rlll rll@{}}
    \toprule
    & \multicolumn{4}{c}{GBM} & \multicolumn{3}{c}{VG} \\
    \cmidrule(lr){2-5}\cmidrule(lr){6-8}
    $d$ & Offline & Density & Price & Sobol hw
        & Offline & Price & Sobol hw \\
    \midrule
    4   & $0.2$\,s  & $2.6\times10^{-8}$ & $7.5\times10^{-7}$ & $1.4\times10^{-6}$
        & $0.8$\,s  & $2.2\times10^{-6}$ & $1.5\times10^{-5}$ \\
    20  & $3.5$\,s  & $2.1\times10^{-6}$ & $5.8\times10^{-6}$ & $7.0\times10^{-6}$
        & $23$\,s   & $8.7\times10^{-6}$ & $5.8\times10^{-5}$ \\
    70  & $85$\,s   & $1.5\times10^{-4}$ & $1.7\times10^{-5}$ & $5.7\times10^{-5}$
        & $505$\,s  & $3.7\times10^{-5}$ & $4.2\times10^{-5}$ \\
    100 & $231$\,s  & $1.3\times10^{-5}$ & $2.0\times10^{-6}$ & $6.2\times10^{-5}$
        & $1361$\,s & $1.0\times10^{-5}$ & $5.8\times10^{-5}$ \\
    150 & $745$\,s  & $2.0\times10^{-5}$ & $2.5\times10^{-5}$ & $8.9\times10^{-5}$
        &           &                    & \\
    \bottomrule
  \end{tabular}
\end{table}

The online curves coincide because both methods supply the same tensorized
representation of the Fourier--cosine expansion to the same online stage. At
GBM $d = 10$, one online evaluation costs $0.57$\,s for COS-TT-CHF and
$0.52$\,s for COS-TT. Online time grows approximately linearly, as $d^{1.04}$
for GBM and $d^{1.19}$ for VG, reaching $8.96$\,s at GBM $d = 150$. This growth
is modest relative to the four orders of magnitude spanned by the offline
times. COS-TT-CHF provides a further structural saving because neither $\bK$
nor $\lambda_1$ enters the ch.f.\ compression. These parameters enter only the
core-wise map applied to the compressed tensor. Evaluating $N$ combinations of
$(K,\lambda_1)$ therefore requires one compression and $N$ maps for COS-TT-CHF,
but $N$ compressions for COS-TT. For the $45$ combinations in
Figure~\ref{fig:k-l1} at VG $d = 6$, the respective counts are
$8.5\times10^{7}$ and $1.04\times10^{10}$, a factor of $123$.

In Figure~\ref{fig:dimension} all controls are held fixed and only the number
of dimensions varies up to $d = 50$, where the exact multivariate Normal
reference remains available. The density error grows smoothly and algebraically.
Least-squares fits of log error against log dimension give exponents of $1.6$,
$2.5$, and $1.2$ for $K = 48$, $64$, and $96$, with $R^2$ between $0.92$ and
$0.97$, confirming the algebraic error growth stated in
Section~\ref{sec:intro}. Over the same range, the full coefficient tensor
$\calA$ grows by a factor of $K^{46}$.
The nearly flat $K = 96$ curve shows that the growth at
lower $K$ comes from series truncation rather than loss of low-rank structure.
The basket-price error is nearly constant across all dimensions and term counts,
remaining at the $\lambda_1 = 7$ truncation floor identified in
Figure~\ref{fig:k-l1}. At these controls, the price is therefore limited by
$\mathbf{D}$ rather than $\bK$ or $d$. The entire panel lies below the
quasi-Monte Carlo half-width, so this comparison uses the converged reference
value. Neither method reaches the rank cap of $64$, and the selected ranks remain
between $37$ and $40$ for dimensions $4$--$50$.

\begin{figure}[!htbp]
  \centering
  \includegraphics[width=0.95\textwidth]{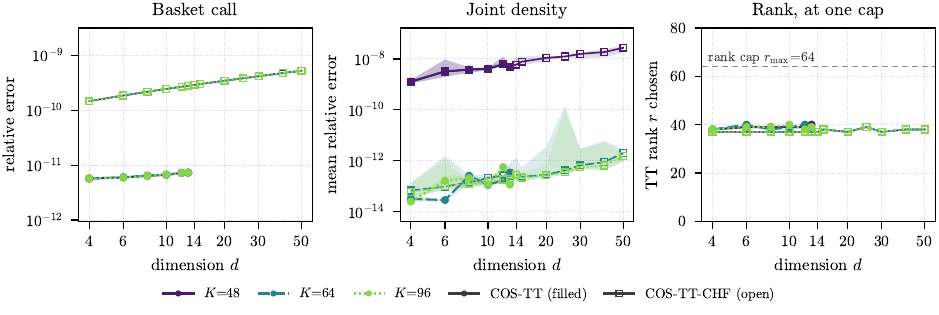}
  \caption{Relative error against dimension for GBM with every control fixed at
    $\Omega = 80$, $\lambda_1 = 7$ and a rank cap of $64$, with $K = 48$, $64$
    and $96$ encoded by colour and dash pattern. \emph{Left:} the basket call.
    \emph{Centre:} the joint density, with bands giving the spread over four
    independent test samples. The bands widen where the curves reach the
    $10^{-13}$ floor, since the mean relative error is then set by a few tail
    draws and depends on the sample. \emph{Right:} the rank selected under
    the common cap, drawn as a dashed rule. COS-TT ends at $d = 13$, where a
    single compression exceeds the time budget.}
  \label{fig:dimension}
\end{figure}

\paragraph{Parameter selection in practice.}
Together, these results give a practical order for selecting the controls.
Choose $\mathbf{D}$, equivalently $\lambda_1$, from the price or mass defect
rather than the density error alone, and choose $\bK$ from coefficient decay.
Set the truncation range by \eqref{eq:Omega-compatibility}, using
$\Omega_j=1.2\kappa_j$ as a safe default or $1.5\kappa_j$ conservatively, and set
$M_j\geq1.3\,\chi_j$. Finally, give DMRG-greedy a sufficiently large rank cap
and let it select the rank. A binding cap indicates that it should be increased.
The resulting representation is reusable across strikes and payoffs. COS-TT-CHF also reuses a
single ch.f.\ compression across the entire $(K,\lambda_1)$ grid, making
parameter studies substantially less expensive than with COS-TT.

\section{Conclusions and future work}
\label{sec:conclusion}

Our principal contribution is COS-TT-CHF, a tensorized COS method that,
together with its parameter-selection rules, breaks the curse of
dimensionality for the problems studied here.

For fixed numbers of cosine terms and bounded TT ranks, and at fixed cost
per ch.f.\ evaluation, both the offline TT-core construction and the
online density-recovery and expectation computations scale linearly with
dimension, while the density error grows only algebraically. Measured
runtimes grow somewhat faster because a single ch.f.\ evaluation costs
$\mathcal{O}(d^2)$ for the models considered here. The experiments reach
$d=150$ for GBM and $d=100$ for VG with good accuracy and within a
$30$-minute offline computation budget, despite the different smoothness
properties of the two densities.

COS-TT provides a more direct complementary method. Its online cost also
scales linearly with dimension, but its offline construction remains
exponential because each COS coefficient requires $2^{d-1}$ ch.f.
evaluations. Its simpler implementation and standard COS error control
nevertheless make it useful in moderate dimensions, with computations
reported up to dimension $26$ \citep{schaap2025costt}.

Both methods share the expectation-calculation approach of
Section~\ref{sec:basket}. For a target function of a scalar aggregate,
contracting the compressed density against a separable exponential
reduces the $d$-dimensional expectation to a one-dimensional COS
calculation. For weighted sums of asset prices, the contraction
integrals are available in closed form
(Theorem~\ref{thm:inner-integrals}). Because the payoff is never
tensorized, a single compressed density can be reused across many payoffs
and strikes.

\paragraph{Future work.}

These methods, and COS-TT-CHF in particular, suggest applications to more
challenging modelling problems. Candidates include counterparty credit
risk, where exposure profiles involve expectations of functions of
weighted sums of dependent risk factors, and multi-asset option pricing
beyond standard baskets, including contingent and exotic features.

More broadly, COS-TT-CHF opens a route to revisiting Fourier-based methods
for high-dimensional problems that were previously out of reach.

Sharper error control and more efficient frequency-domain discretization
remain important directions. Proposition~\ref{prop:Omega-selection}
converts an assumed algebraic ch.f.\ tail into a frequency-domain
truncation range, whereas our quadrature results provide resolution
conditions rather than direct accuracy bounds. Deriving model-specific
tail constants and oscillation scales would turn both into explicit
bounds on $\varepsilon_{\bm{\Omega}}$ and
$\varepsilon_{\mathrm{quad}}$, and would allow $\bm{\Omega}$ and
$\mathbf{M}$ to be optimized jointly. In parallel, quadrature rules
tailored to the analytic oscillatory factors of the kernel may reduce
$\mathbf{M}$ while preserving the core-wise TT construction, as
discussed in Remark~\ref{rem:oscillatory-quadrature}.

\appendix

\section{Error analysis of COS-TT}
\label{sec:error}

The three arrows of the approximation chain \eqref{eq:approx-chain}, established in
Section~\ref{subsec:COS-TT}, give three additive error components:
\begin{equation}
   \label{eq:error-decomp}
   \norm{f - \hat{f}_{\mathbf{r}\text{-CTT}}}_{L^2}
   \leq \underbrace{\norm{f - \tilde{f}}_{L^2}}_{\varepsilon_{\bK}}
   + \underbrace{\norm{\tilde{f} - \hat{f}}_{L^2}}_{\varepsilon_{\mathrm{COS}}}
   + \underbrace{\norm{\hat{f} - \hat{f}_{\mathbf{r}\text{-CTT}}}_{L^2}}_{\varepsilon_{\CTT}}.
\end{equation}
The first two components belong to the underlying COS method. Bounds for both in
the multidimensional setting, and the resulting trade-off between the
physical-domain truncation range and the expansion-term count, are given in
\citet[Thm.~3.7 and Prop.~A.1]{junike2025characteristic}. We do not repeat them here.

\subsection{TT decomposition error
\texorpdfstring{$\varepsilon_{\CTT}$}{eps\_CTT}}
\label{subsec:eps-TT}

By \eqref{eq:error-decomp}, $\varepsilon_{\CTT}$ is the $L^2(\mathbf{D})$ norm
of the density error $\hat{f} - \hat{f}_{\mathbf{r}\text{-CTT}}$.
It includes both rank truncation and the numerical error of the decomposition
algorithm.

The cores are computed by the DMRG-greedy scheme in
Section~\ref{subsec:TT-algorithms}. A complete convergence theory for this
rank-adaptive scheme is beyond the scope of this paper. We therefore do not
bound $\varepsilon_{\CTT}$ in terms of the method parameters. Instead, we assume
that the decomposition attains a prescribed coefficient-tensor accuracy and
derive the COS-specific conversion to density error. The other two components
are bounded explicitly, so the total accuracy is conditional only on the
decomposition tolerance.

By \eqref{eq:fhat}, the retained COS density is
$\hat{f}(\bx) = \sum_{\bk}^{\bK}\mathcal{A}_{\bk}\,\Phi_{\bk}(\bx)$. The
decomposition stage approximates $\mathcal{A}$ on $\IK$ by the rank-$\mathbf{r}$
reconstruction $\hat{\mathcal{A}}$ of \eqref{eq:A-TT}, assembled from computed
cores $\hat{A}_j \in \R^{r_{j-1}\times K_j\times r_j}$. Writing
$\Delta\mathcal{A}_{\bk} := \mathcal{A}_{\bk} - \hat{\mathcal{A}}_{\bk}$ for the
coefficient error, it follows that
\begin{equation}
  \label{eq:eps-train-series}
  \hat{f}(\bx) - \hat{f}_{\mathbf{r}\text{-CTT}}(\bx)
  = \sum_{\bk}^{\bK}\Delta\mathcal{A}_{\bk}\,\Phi_{\bk}(\bx).
\end{equation}

\begin{theorem}[Coefficient-to-density conversion]
  \label{thm:train-conversion}
  The decomposition error of the reconstructed density equals the Frobenius-norm
  error of the approximated coefficient tensor,
  \begin{equation}
    \label{eq:rank-error-identity}
    \varepsilon_{\CTT}
    := \norm{\hat{f} - \hat{f}_{\mathbf{r}\text{-CTT}}}_{L^2(\mathbf{D})}
    = \norm{\mathcal{A} - \hat{\mathcal{A}}}_F .
  \end{equation}
\end{theorem}

\begin{proof}
  By orthonormality of $\{\Phi_{\bk}\}$, Parseval's identity applied to
  \eqref{eq:eps-train-series} gives
  $\varepsilon_{\CTT}^2
  = \norm{\hat{f} - \hat{f}_{\mathbf{r}\text{-CTT}}}_{L^2(\mathbf{D})}^2
  = \sum_{\bk}^{\bK}|\Delta\mathcal{A}_{\bk}|^2
  = \norm{\Delta\mathcal{A}}_F^2$.
  The modulus is written because Section~\ref{sec:error-chf} reuses the
  identity for coefficient tensors that need not be real.
\end{proof}

The proof uses only orthonormality of the product cosine basis, so the same
equality holds for any pair of coefficient tensors expanded in that basis.

\begin{assumption}[Decomposition regime]
  \label{ass:training}
  For a tolerance $\delta_{\TT}>0$, the decomposition stage returns cores whose
  reconstruction \eqref{eq:A-TT} satisfies
  $\norm{\mathcal{A}-\hat{\mathcal{A}}}_F \leq \delta_{\TT}$.
\end{assumption}

This assumption concerns the decomposition stage rather than the COS
construction: the required tolerance $\delta_{\TT}$ is not expressed in terms
of $K$, $\mathbf{D}$, or $d$. The resulting accuracy is therefore conditional on
the decomposition stage but explicit in the two COS parameters.

Assumption~\ref{ass:training} is not verified directly in the reported
high-dimensional experiments because the full tensors cannot be formed.
Instead, the rank-convergence study in Section~\ref{subsec:num-K-rank} assesses
its practical effect through the resulting density and expectation errors.
Whenever the full tensor is available, the Frobenius error can be evaluated
directly.

Available theory also bounds this quantity in the intended regime. For cross
interpolation, the error is the best rank-$\mathbf{r}$ error multiplied by an
amplification factor that grows polynomially rather than exponentially in $d$
\citep{Savostyanov2014,Osinsky2019,qin2022error}. The remaining constants depend on
the quality of the interpolation sets, which greedy pivoting does not control. We
take these results as they stand and do not specialize them to the TT-DMRG-greedy
variant used here.

\section*{Acknowledgements}
\addcontentsline{toc}{section}{Acknowledgements}

We thank FF Quant Advisory B.V.\ and Dr. Xiaoyu Shen from the company for supporting a series of 
proof-of-concept studies on COS-tensor methods, carried out as MSc theses
(2021--2025), which laid the foundation for this work.

\section*{Statements and Declarations}
\addcontentsline{toc}{section}{Statements and Declarations}

\subsection*{Competing Interests}

Fang Fang is employed by FF Quant Advisory B.V.\ and by the Delft Institute 
of Applied Mathematics, Delft University of Technology.
 As stated in the Acknowledgements, FF Quant Advisory B.V.\ supported the MSc
proof-of-concept studies that preceded and informed this work. Auke Schaap is
affiliated with the Delft Institute of Applied Mathematics, Delft University of
Technology. The company had no role in the design of the methods, the analysis,
or the decision to publish.
The authors have no other financial or non-financial interests, direct or
indirect, that are related to the work submitted for publication. The views
expressed in this paper are those of the authors and do not necessarily
reflect the views of their affiliated institutions.

\subsection*{Funding}

The authors did not receive support from any organization for the submitted
work.

\subsection*{Data and Code Availability}

No datasets were generated or analysed in this study. All numerical
experiments use synthetic model parameters, reported in full in
Section~\ref{sec:numerics}. The research code implementing COS-TT and
COS-TT-CHF is not publicly available at present. It is available from the
corresponding author on reasonable request, and we intend to release it on
publication.

\subsection*{Declaration of AI Use}

The methodology, the theoretical results, and their implementation are our own.
No large language model (LLM) is an author of this paper. We used LLM-based
assistants (Claude Code and ChatGPT-5) to debug and refactor
code, to automate tests, to cross-check derivations we had already obtained by
hand, and to check references. They were also used for copy editing, that is,
for improving readability, grammar and style, but not for generative editorial
work or autonomous content creation. We derived and verified every result
ourselves and reviewed all such output. The authors take full responsibility for
the content of this paper and confirm that the final text reflects their own work.

\bibliographystyle{abbrvnat}
\bibliography{literature,costt_extra}

@article{duffie2002affine,
  author  = {Duffie, Darrell and Filipovi\'c, Damir and Schachermayer, Walter},
  title   = {Affine processes and applications in finance},
  journal = {The Annals of Applied Probability},
  volume  = {13},
  number  = {3},
  pages   = {984--1053},
  year    = {2003},
  doi     = {10.1214/aoap/1060202833}
}

@article{ruijter2012two,
  author    = {Ruijter, Marjon J. and Oosterlee, Cornelis W.},
  title     = {Two-Dimensional {Fourier} Cosine Series Expansion Method for Pricing Financial Options},
  journal   = {SIAM Journal on Scientific Computing},
  volume    = {34},
  number    = {5},
  pages     = {B642--B671},
  year      = {2012},
  publisher = {SIAM},
  doi       = {10.1137/120862053}
}

@misc{kastoryano2022highly,
  author        = {Kastoryano, Michael and Pancotti, Nicola},
  title         = {A Highly Efficient Tensor Network Algorithm for Multi-Asset {Fourier} Options Pricing},
  year          = {2022},
  eprint        = {2203.02804},
  archivePrefix = {arXiv},
  primaryClass  = {quant-ph},
  url           = {https://arxiv.org/abs/2203.02804}
}

@article{white1992density,
  author  = {White, Steven R.},
  title   = {Density Matrix Formulation for Quantum Renormalization Groups},
  journal = {Physical Review Letters},
  volume  = {69},
  number  = {19},
  pages   = {2863--2866},
  year    = {1992},
  doi     = {10.1103/PhysRevLett.69.2863}
}

@article{schollwock2011density,
  author  = {Schollw\"{o}ck, Ulrich},
  title   = {The Density-Matrix Renormalization Group in the Age of Matrix Product States},
  journal = {Annals of Physics},
  volume  = {326},
  number  = {1},
  pages   = {96--192},
  year    = {2011},
  doi     = {10.1016/j.aop.2010.09.012}
}

@book{hackbusch2019tensor,
  author    = {Hackbusch, Wolfgang},
  title     = {Tensor Spaces and Numerical Tensor Calculus},
  series    = {Springer Series in Computational Mathematics},
  volume    = {56},
  publisher = {Springer},
  address   = {Cham},
  edition   = {2nd},
  year      = {2019},
  doi       = {10.1007/978-3-030-35554-8},
  isbn      = {9783030355548}
}

@inproceedings{qin2022error,
  author    = {Qin, Zhen and Lidiak, Alexander and Gong, Zhexuan and Tang, Gongguo
               and Wakin, Michael B. and Zhu, Zhihui},
  title     = {Error analysis of tensor-train cross approximation},
  booktitle = {Advances in Neural Information Processing Systems 35 (NeurIPS 2022)},
  pages     = {14236--14249},
  year      = {2022},
  doi       = {10.52202/068431-1035},
  eprint    = {2207.04327},
  archivePrefix = {arXiv},
  primaryClass  = {cs.LG}
}

@article{Osinsky2019,
  author  = {Osinsky, A. I.},
  title   = {Tensor trains approximation estimates in the {C}hebyshev norm},
  journal = {Computational Mathematics and Mathematical Physics},
  volume  = {59},
  number  = {2},
  pages   = {201--206},
  year    = {2019},
  doi     = {10.1134/S096554251902012X}
}

@misc{arenstein2026costtchf,
  author        = {Arenstein, Lucas and Kastoryano, Michael},
  title         = {{COS--TT--CHF}: A Tensor-Train Characteristic-Function {COS} Method
                   for Multi-Asset Option Pricing},
  year          = {2026},
  eprint        = {2608.17636},
  archivePrefix = {arXiv},
  primaryClass  = {q-fin.CP},
  url           = {https://arxiv.org/abs/2608.17636}
}

@article{MastFang2026,
  author  = {Mast, Gijs and Fang, Fang and Shen, Xiaoyu and Brands, Marnix},
  title   = {Dimension-Reduced {Fourier--Cosine} Expansion via Canonical Polyadic Decomposition with Application to Counterparty Credit Exposures},
  journal = {IMA Journal of Numerical Analysis},
  year    = {2026},
  note    = {Submitted}
}

@article{trefethen2008gauss,
  author  = {Trefethen, Lloyd N.},
  title   = {Is {G}auss Quadrature Better than {C}lenshaw--{C}urtis?},
  journal = {SIAM Review},
  volume  = {50},
  number  = {1},
  pages   = {67--87},
  year    = {2008},
  doi     = {10.1137/060659831}
}

@book{deano2018computing,
  author    = {Dea{\~n}o, Alfredo and Huybrechs, Daan and Iserles, Arieh},
  title     = {Computing Highly Oscillatory Integrals},
  publisher = {SIAM},
  address   = {Philadelphia},
  year      = {2018},
  doi       = {10.1137/1.9781611975123},
  isbn      = {9781611975116}
}

@book{atkinson1989introduction,
  author    = {Atkinson, Kendall E.},
  title     = {An Introduction to Numerical Analysis},
  edition   = {2},
  publisher = {John Wiley \& Sons},
  address   = {New York},
  year      = {1989},
  isbn      = {9780471624899}
}

@article{luciano2006multivariate,
  author  = {Luciano, Elisa and Schoutens, Wim},
  title   = {A Multivariate Jump-Driven Financial Asset Model},
  journal = {Quantitative Finance},
  volume  = {6},
  number  = {5},
  pages   = {385--402},
  year    = {2006},
  doi     = {10.1080/14697680600806275}
}

@book{cont2004financial,
  author    = {Cont, Rama and Tankov, Peter},
  title     = {Financial Modelling with Jump Processes},
  series    = {Chapman \& Hall/CRC Financial Mathematics Series},
  publisher = {Chapman \& Hall/CRC},
  address   = {Boca Raton, FL},
  year      = {2004},
  isbn      = {9781584884132}
}

@article{bigoni2016spectral,
  title = {Spectral Tensor-Train Decomposition},
  author = {Bigoni, Daniele and Engsig-Karup, Allan P. and Marzouk, Youssef M.},
  journal = {SIAM Journal on Scientific Computing},
  volume = {38},
  number = {4},
  pages = {A2405--A2439},
  year = {2016},
  doi = {10.1137/15M1036919},
  publisher = {SIAM}
}

@article{gorodetsky2019continuous,
  title = {A continuous analogue of the tensor-train decomposition},
  author = {Gorodetsky, Alex and Karaman, Sertac and Marzouk, Youssef},
  journal = {Computer Methods in Applied Mechanics and Engineering},
  volume = {347},
  pages = {59--84},
  year = {2019},
  doi = {10.1016/j.cma.2018.12.015},
  issn = {0045-7825},
  publisher = {Elsevier}
}

@article{fang2009novel,
  title={A novel pricing method for {E}uropean options based on {F}ourier-cosine series expansions},
  author={Fang, Fang and Oosterlee, Cornelis W},
  journal={SIAM Journal on Scientific Computing},
  volume={31},
  number={2},
  pages={826--848},
  year={2008},
  publisher={SIAM},
  doi = {10.1137/080718061}
}

@article{ruijter2015application,
  title={On the application of spectral filters in a {F}ourier option pricing technique},
  author={Ruijter, Marjon and Versteegh, Mark and Oosterlee, Cornelis W},
  journal={Journal of Computational Finance},
  volume={19},
  number={1},
  pages={75--106},
  year={2015},
  doi={10.21314/JCF.2015.306}
}

@article{glau2020low,
	title={Low-rank tensor approximation for {C}hebyshev interpolation in parametric option pricing},
	author={Glau, Kathrin and Kressner, Daniel and Statti, Francesco},
	journal={SIAM Journal on Financial Mathematics},
	volume={11},
	number={3},
	pages={897--927},
	year={2020},
	publisher={SIAM},
	doi = {10.1137/19M1244172}
}

@article{xiang2012convergence,
	author={Xiang, Shuhuang and Bornemann, Folkmar},
	title={On the Convergence Rates of {G}auss and {C}lenshaw--{C}urtis Quadrature for Functions of Limited Regularity},
	journal={SIAM Journal on Numerical Analysis},
	volume={50},
	number={5},
	pages={2581--2587},
	year={2012},
	doi={10.1137/120869845}
}

@article{de2008tensor,
	author={de Silva, Vin and Lim, Lek-Heng},
	title={Tensor rank and the ill-posedness of the best low-rank approximation problem},
	journal={SIAM Journal on Matrix Analysis and Applications},
	volume={30},
	number={3},
	pages={1084--1127},
	year={2008},
	publisher={SIAM},
	doi={10.1137/06066518X}
}

@article{DolgovSavostyanov2019,
  author       = {Dolgov, S. and Savostyanov, D.},
  title        = {Parallel cross interpolation for high-precision calculation of high-dimensional integrals},
  journal      = {Computer Physics Communications},
  volume       = {246},
  pages        = {106869},
  year         = {2020},
  doi          = {10.1016/j.cpc.2019.106869},
  url          = {https://doi.org/10.1016/j.cpc.2019.106869}
}

@article{Savostyanov2014,
  author       = {Savostyanov, D. V.},
  title        = {Quasioptimality of maximum-volume cross interpolation of tensors},
  journal      = {Linear Algebra and its Applications},
  volume       = {458},
  pages        = {217--244},
  year         = {2014},
  doi          = {10.1016/j.laa.2014.06.006},
  url          = {https://doi.org/10.1016/j.laa.2014.06.006}
}

@article{Kargas2021,
    title = {{Supervised Learning and Canonical Decomposition of Multivariate Functions}},
    year = {2021},
    journal = {IEEE Transactions on Signal Processing},
    author = {Kargas, Nikos and Sidiropoulos, Nicholas D.},
    pages = {1097--1107},
    volume = {69},
    url = {https://ieeexplore.ieee.org/document/9340610/},
    doi = {10.1109/TSP.2021.3055000},
    issn = {1053-587X}
}

@article{amiridi2022lowrank,
    title = {{Low-Rank Characteristic Tensor Density Estimation Part I: Foundations}},
    author = {Amiridi, Magda and Kargas, Nikos and Sidiropoulos, Nicholas D.},
    journal = {IEEE Transactions on Signal Processing},
    volume = {70},
    pages = {2654--2668},
    year = {2022},
    doi = {10.1109/TSP.2022.3175608},
    issn = {1053-587X}
}

@article{sakurai2025learning,
  title={Learning Parameter Dependence for {F}ourier-Based Option Pricing with Tensor Trains},
  author={Sakurai, Rihito and Takahashi, Haruto and Miyamoto, Koichi},
  journal={Mathematics},
  volume={13},
  number={11},
  pages={1828},
  year={2025},
  doi={10.3390/math13111828},
  publisher={MDPI}
}

@article{oseledets2011tensor,
  title={Tensor-train decomposition},
  author={Oseledets, Ivan V},
  journal={SIAM Journal on Scientific Computing},
  volume={33},
  number={5},
  pages={2295--2317},
  year={2011},
  publisher={SIAM},
  doi = {10.1137/090752286}
}

@article{OseledetsTyrtyshnikov2010,
  author  = {Oseledets, Ivan V. and Tyrtyshnikov, Eugene E.},
  title   = {TT-cross approximation for multidimensional arrays},
  journal = {Linear Algebra and its Applications},
  volume  = {432},
  number  = {1},
  pages   = {70--88},
  year    = {2010},
  doi     = {10.1016/j.laa.2009.07.024}
}

@article{hackbusch2009new,
  title={A New Scheme for the Tensor Representation},
  author={Hackbusch, Wolfgang and K{\"u}hn, Stefan},
  journal={Journal of Fourier Analysis and Applications},
  volume={15},
  number={5},
  pages={706--722},
  year={2009},
  doi={10.1007/s00041-009-9094-9}
}

@article{junike2025characteristic,
	author={Junike, Gero and Stier, Hauke},
	title={From Characteristic Functions to Multivariate Distribution Functions and {E}uropean Option Prices by the (Damped) {COS} Method},
	journal={SIAM Journal on Numerical Analysis},
	volume={63},
	number={6},
	pages={2421--2453},
	year={2025},
	publisher={SIAM},
	doi={10.1137/24M1666240}
}

@mastersthesis{schaap2025costt,
  author = {Schaap, Auke},
  title  = {Dimension-Reduced {Fourier} Cosine Series Expansion Based on {Tensor Train} Decomposition and Its Application in Finance},
  school = {Delft University of Technology},
  type   = {{MSc} thesis},
  year   = {2025},
  address = {Delft, the Netherlands},
  url    = {https://diamhomes.ewi.tudelft.nl/~kvuik/numanal/schaap_afst.pdf}
}

\end{document}